\documentclass[11pt, letterpaper, oneside,reqno]{amsart}

\usepackage{latexsym, amsmath, amssymb, amsfonts, amscd,bm}
\usepackage{amsthm}
\usepackage{t1enc}
\usepackage[mathscr]{eucal}
\usepackage{indentfirst}
\usepackage{graphicx, pb-diagram}
\usepackage{fancyhdr}
\usepackage{fancybox}
\usepackage{enumerate}
\usepackage{color}
\usepackage[all]{xy}
\usepackage{hyperref}
\usepackage{tikz-cd}
\usepackage{mathtools}
\usepackage{blkarray}
\usepackage{url}

\theoremstyle{plain}
\newtheorem{thm}{Theorem}[section]
\newtheorem*{thmnonumber}{Theorem}

\newtheorem{prop}[thm]{Proposition}
\newtheorem*{propnonumber}{Proposition}

\newtheorem{lemma}[thm]{Lemma}
\newtheorem{cor}[thm]{Corollary}

\newtheoremstyle{underline}
{}        
{}              
{}              
{}    
{\large}              
{:}             
{1mm}         
{{\underline{\thmname{#1}\thmnumber{ #2}}}}  

\theoremstyle{underline}
\newtheorem*{claim*}{Claim}
\newtheorem{claim}{Claim}

\theoremstyle{definition}
\newtheorem{defi}[thm]{Definition}

\theoremstyle{remark}
\newtheorem{remark}[thm]{Remark}
\newtheorem{ex}[thm]{Example}
\newtheorem{exs}[thm]{Examples}

\newcommand{\CC}{\ensuremath{\mathbb C}}
\newcommand{\RR}{\ensuremath{\mathbb R}}
\newcommand{\g}{\ensuremath{\mathfrak{g}}}

\newcommand*\conj[1]{\bar{#1}}

\newcommand{\pd}[1]{{\partial_{#1}}} 
 
\definecolor{forest}{rgb}{0,0.5,0}

\newcommand{\squeezeup}{\vspace{-7mm}}

\begin{document}

\title{Coisotropic submanifolds in  $b$-symplectic geometry}

\author{Stephane Geudens}
\address{KU Leuven, Department of Mathematics, Celestijnenlaan 200B box 2400, BE-3001 Leuven, Belgium
 }
\email{stephane.geudens@kuleuven.be}

\author{Marco Zambon}
\email{marco.zambon@kuleuven.be}


\begin{abstract}
We study coisotropic submanifolds of $b$-symplectic manifolds. 
We prove that  $b$-coisotropic submanifolds (those transverse to the degeneracy locus) determine the $b$-symplectic structure in a neighborhood, and provide a normal form theorem. This extends Gotay's theorem in symplectic geometry. 
Further, we introduce strong $b$-coisotropic submanifolds and show that their coisotropic quotient, which locally is always smooth, inherits a reduced $b$-symplectic structure.
\end{abstract}

\maketitle

\setcounter{tocdepth}{1} 
\tableofcontents

\section*{Introduction}

In symplectic geometry, an important and interesting class of submanifolds are the coisotropic ones. They are the submanifolds $C$ satisfying $TC^{\Omega}\subset TC$, where $TC^{\Omega}$ denotes the symplectic orthogonal of the tangent bundle $TC$. They arise for instance as zero level sets of moment maps, and in mechanics as those submanifolds that are given by first class constraints (see Dirac's theory of constraints). The notion of coisotropic submanifolds extends to the wider realm of Poisson geometry, and it plays an important role there too: for instance, a map is a  Poisson morphism {if and only if} its graph is coisotropic, and coisotropic submanifolds admit canonical quotients which inherit a Poisson structure.

The Poisson structures  which are non-degenerate at every point are exactly the symplectic ones. {Relaxing slightly the non-degeneracy condition, one obtains Poisson structures $(M,\Pi)$ for which the top power $\wedge^n \Pi$  is   transverse to the zero section of the line bundle $\wedge^{2n} TM$ (here $\dim(M)=2n$): they are called \emph{log-symplectic} structures}. 
They are symplectic outside {the vanishing set of $\wedge^n \Pi$,  a hypersurface} which inherits a codimension-one symplectic foliation. Log-symplectic structures are studied systematically {by Guillemin-Miranda-Pires} in \cite{miranda}, and turn out to be equivalent to \emph{$b$-symplectic} structures. The latter are defined on manifolds $M$ with a choice of codimension-one submanifold $Z$, as follows: they are non-degenerate sections $\omega$ of $\wedge^2 ({}^{b}TM)^*$ which are closed w.r.t. the de Rham differential, where ${}^{b}TM$ is the $b$-tangent bundle (a Lie algebroid over $M$ which encodes $Z$). In other words, they are the analogue of symplectic forms if one replaces the tangent bundle  with the $b$-tangent bundle. Because of this, various phenomena in symplectic geometry have counterparts for log-symplectic manifolds.

{This paper is devoted to coisotropic submanifolds of log-symplectic manifolds.}
We single out two classes, which we call \emph{$b$-coisotropic} and \emph{strong $b$-coisotropic}.
We prove that certain properties of coisotropic submanifolds in symplectic geometry {-- 
properties which certainly do not carry over to arbitrary coisotropic submanifolds of log-symplectic manifolds -- do }
carry over to the above classes. {Moreover, we show that these classes of submanifolds enjoy some properties that are $b$-geometric enhancements of well-known facts about coisotropic submanifolds in Poisson geometry.} We now elaborate on this.

\bigskip
\textbf{Main results.}  
Let $(M,Z,\omega)$ be a $b$-symplectic manifold, and denote by $\Pi$ the corresponding Poisson tensor on $M$.
We consider two classes of  submanifolds which are coisotropic {(in the sense of Poisson geometry)} with respect to $\Pi$.

A submanifold  of $M$ is called \emph{$b$-coisotropic} if it is coisotropic   and {a $b$-submanifold (i.e. transverse to $Z$)}.
An equivalent characterization is the following: a 
$b$-submanifold $C$ such that $({}^{b}TC)^\omega\subset{}^{b}TC$. 
{The latter} formulation makes apparent that this notion is very natural in $b$-symplectic geometry. 
 Section \ref{sec:bcoiso} is devoted to {the class of $b$-coisotropic} submanifolds.

We show that the $b$-conormal bundle of a $b$-coisotropic submanifold is a Lie subalgebroid. We also show  that for Poisson maps between {log-symplectic} manifolds compatible with the corresponding hypersurfaces, the graphs are $b$-coisotropic submanifolds, once ``lifted'' to a suitable blow-up \cite{GuLi}. Both of these statements are $b$-geometric analogs of well-known facts about coisotropic submanifolds in Poisson geometry.
Next, in Theorem \ref{gotay} we show that Gotay's theorem in symplectic geometry \cite{Gotay}
extends to $b$-coisotropic submanifolds in $b$-symplectic geometry. The main consequence is a   normal form theorem for the $b$-symplectic structure around such submanifolds:

\begin{thmnonumber} A neighborhood of a $b$-coisotropic  submanifold $C\overset{i}{\hookrightarrow}(M,Z,\omega)$ 
is $b$-symplectomorphic to the following model: 
$$(\text{a neighborhood of the zero section in $E^*$, $\Omega$}),$$
{where the vector bundle $E:=\ker(^bi^*\omega)$ denotes the kernel of the pullback of $\omega$ to $C$, and $\Omega$ is a $b$-symplectic  form   which is constructed out of the pullback $^bi^*\omega$ and is canonical up to
neighborhood equivalence (see equation \eqref{Omega} for the precise formula).}
\end{thmnonumber}
{Such a normal form} allows to study effectively the deformation theory of $C$ as a coisotropic submanifold \cite{GZbdefs}. Another possible application is the construction of $b$-symplectic manifolds using surgeries, as done for instance in \cite[Theorem 6.1]{ExamGil}. We point out that in the special case of Lagrangian submanifolds,  the above result is a version of Weinstein's tubular neighborhood theorem, and was already obtained {by Kirchhoff-Lukat} \cite[Theorem 5.18]{CharlotteThesis}.

\bigskip
In Section \ref{sec:strongbcoiso} we consider the following subclass of the $b$-coisotropic submanifolds. 
A submanifold $C$   is called  \emph{strong $b$-coisotropic} if it is  coisotropic   and transverse to all the symplectic leaves of $ (M,\Pi)$ it meets. We remark that Lagrangian submanifolds intersecting the degeneracy hypersurface $Z$ never satisfy this definition.

The main feature of strong $b$-coisotropic submanifolds is that the characteristic distribution $$D:=\Pi^{\sharp}\left(TC^{0}\right),$$
is \emph{regular}, with rank equal to $codim(C)$.
{Recall the following   fact in Poisson geometry:} when the quotient of a coisotropic submanifold by its characteristic distribution is a smooth manifold, then it inherits a Poisson structure, called the reduced Poisson structure.
We  show (see Proposition \ref{prop:bcoisored} for the full statement):
\begin{propnonumber}
Let $C$ be a strong $b$-coisotropic submanifold of a $b$-symplectic manifold. If the quotient $C/D$ by the characteristic distribution is smooth, then the reduced Poisson structure is again $b$-symplectic.
\end{propnonumber}

Instances of the above proposition arise when a {connected} Lie group acts on a $b$-symplectic manifold with equivariant moment map, in the sense of Poisson geometry, and
$C$ is the zero level set of the latter, see
Corollary \ref{reduction}. At the end of the paper we provide  examples of   $b$-symplectic 
quotients, and -- by  reversing the procedure -- in Corollary \ref{cor:explh} we realize any $b$-symplectic structure on the 2-dimensional sphere as such a quotient. 
\bigskip

In order to state and prove these results, in Section \ref{sec:background} we collect some  facts about $b$-geometry.
A few of them are new, {to the best of our knowledge,} and are of independent interest.
More specifically, in Lemma \ref{lem:distrsplit} we  show  that, while the anchor map of the $b$-tangent bundle does not admit a canonical splitting,  distributions tangent to $Z$ do have a canonical lift to the $b$-tangent bundle. In Proposition \ref{moser} we provide a version of the $b$-Moser theorem relative to a $b$-submanifold, which we could not find elsewhere in the literature.

\bigskip
\textbf{Acknowledgements.}  
 We acknowledge partial support by the long term structural funding -- Methusalem grant of the Flemish Government, the FWO under EOS project G0H4518N, the FWO research project G083118N (Belgium).

\section{Background on $b$-geometry} \label{sec:background}
In this section, we address the formalism of $b$-geometry, which originated from work of Melrose \cite{Melrose} in the context of manifolds with boundary.  We review some of the main concepts, including $b$-symplectic structures, and we prove some preliminary results that will be used in the body of this paper.

\subsection{$b$-manifolds and $b$-maps}
\leavevmode
\vspace{0.15cm}

We first introduce the objects and morphisms of the $b$-category, following \cite{miranda}.

\begin{defi}
A \emph{$b$-manifold} is a pair $(M,Z)$ consisting of a manifold $M$ and a codimension-one submanifold $Z\subset M$.	
\end{defi}

Given a $b$-manifold $(M,Z)$, we denote by ${}^{b}\mathfrak{X}(M)$ the set of vector fields on $M$ that are tangent to $Z$. Note that ${}^{b}\mathfrak{X}(M)$ is a locally free $C^{\infty}(M)$-module, with generators
\[
x_{1}\partial_{x_{1}},\partial_{x_{2}},\ldots,\partial_{x_{n}}
\]  
in a coordinate chart $(x_{1},\ldots,x_{n})$ adapted to $Z=\{x_{1}=0\}$. Thanks to the Serre-Swan theorem, these $b$-vector fields give rise to a vector bundle ${}^{b}TM$.

\begin{defi}
Let $(M,Z)$ be a $b$-manifold. The \emph{$b$-tangent bundle} ${}^{b}TM$ is the vector bundle over $M$ satisfying $\Gamma\left({}^{b}TM\right)={}^{b}\mathfrak{X}(M)$.
\end{defi} 

The natural inclusion ${}^{b}\mathfrak{X}(M)\subset \mathfrak{X}(M)$ induces a vector bundle map $\rho:{}^{b}TM\rightarrow TM$, which is an isomorphism away from $Z$. Restricting to $Z$, we get a bundle epimorphism $\rho|_{Z}:{}^{b}TM|_{Z}\rightarrow TZ$, which gives rise to a trivial line bundle $\mathbb{L}:=\text{Ker}\left(\rho|_{Z}\right)$. Indeed, $\mathbb{L}$ is canonically trivialized by the normal $b$-vector field $\xi\in\Gamma(\mathbb{L})$, which is locally given by $x\partial_{x}$ where $x$ is any local defining function for $Z$. So at any point $p\in Z$, we have a short exact sequence
\begin{equation}\label{eq:sesL}
0\rightarrow\mathbb{L}_{p}\hookrightarrow{}^{b}T_{p}M\overset{\rho}{\rightarrow}T_{p}Z\rightarrow 0,
\end{equation} 
but this sequence does not split canonically.

\vspace{0.5cm}

Since ${}^{b}\mathfrak{X}(M)$ is a Lie subalgebra of $\mathfrak{X}(M)$, it inherits a natural Lie bracket $[\cdot,\cdot]$. The data $(\rho,[\cdot,\cdot])$ endow ${}^{b}TM$ with a Lie algebroid structure. The map $\rho$ is called the anchor of ${}^{b}TM$.

\begin{defi}
Let $(M,Z)$ be a $b$-manifold. The \emph{$b$-cotangent bundle} ${}^{b}T^{*}M$ is the dual bundle  of ${}^{b}TM$.	
\end{defi}
In coordinates $(x_{1},\ldots,x_{n})$ adapted to $Z=\{x_{1}=0\}$, the $b$-cotangent bundle ${}^{b}T^{*}M$ has local frame
\[
\frac{dx_{1}}{x_{1}},dx_{2},\ldots,dx_{n}.
\]

We will denote the set $\Gamma\left(\wedge^{k}\left({}^{b}T^{*}M\right)\right)$ of Lie algebroid $k$-forms by ${}^{b}\Omega^{k}(M)$, and we refer to them as $b$-$k$-forms. The space ${}^{b}\Omega^{\bullet}(M)$ is endowed with the Lie algebroid differential ${}^{b}d$, which is determined by the fact that the restriction $\left({}^{b}\Omega^{k}(M),{}^{b}d\right)\rightarrow \left(\Omega^{k}(M\setminus Z),d\right)$ is a chain map. Note that the anchor $\rho$ induces an injective map $\rho^{*}:\Omega^{k}(M)\rightarrow{}^{b}\Omega^{k}(M)$, which allows us to view honest de Rham forms as $b$-forms.

\begin{defi}
Given $b$-manifolds $(M_{1},Z_{1})$ and $(M_{2},Z_{2})$, a \emph{$b$-map} $f:(M_{1},Z_{1})\rightarrow (M_{2},Z_{2})$ is a smooth map $f:M_{1}\rightarrow M_{2}$ such that $f$ is transverse to $Z_{2}$ and $f^{-1}(Z_{2})=Z_{1}$.
\end{defi}

Given a $b$-map $f:(M_{1},Z_{1})\rightarrow (M_{2},Z_{2})$, the usual pullback $f^{*}:\Omega^{\bullet}(M_{2})\rightarrow \Omega^{\bullet}(M_{1})$ extends to an algebra morphism  ${}^{b}f^{*}:{}^{b}\Omega^{\bullet}(M_{2})\rightarrow{}^{b}\Omega^{\bullet}(M_{1})$, see \cite[Proof of Proposition 3.5.2]{RalphThesis}. That is, we have a commutative diagram
\[
\begin{tikzcd}[column sep=large, row sep=large]
{}^{b}\Omega^{\bullet}(M_{2})\arrow{r}{^{b}f^{*}} & {}^{b}\Omega^{\bullet}(M_{1})\\
\Omega^{\bullet}(M_{2})\arrow{u}{\rho_{2}^{*}}\arrow{r}{f^{*}} & \Omega^{\bullet}(M_{1})\arrow{u}{\rho_{1}^{*}}\\
\end{tikzcd}.
\]
\squeezeup

This $b$-pullback has the expected properties; for instance, the assignment $f\mapsto {}^{b}f^{*}$ is functorial, and the $b$-pullback ${}^{b}f^{*}$ commutes with the $b$-differential ${}^{b}d$.

We can now define the Lie derivative of a $b$-form $\omega\in{}^{b}\Omega^{k}(M)$ in {the} direction of a $b$-vector field $X\in{}^{b}\mathfrak{X}(M)$ by the usual formula
\[
\pounds_{X}\omega=\left. \frac{d}{dt}\right|_{t=0}{}^{b}\rho_{t}^{*}\omega,
\]
where the $b$-pullback is well-defined since the flow $\{\rho_{t}\}$ of $X$ consists of $b$-diffeomorphisms. Cartan's formula is still valid
\[
\pounds_{X}\omega={}^{b}d\iota_{X}\omega+\iota_{X}{}^{b}d\omega.
\]

Dual to the $b$-pullback ${}^{b}f^{*}$, a $b$-map $f:(M_{1},Z_{1})\rightarrow (M_{2},Z_{2})$ induces a $b$-derivative ${}^{b}f_{*}:{}^{b}TM_{1}\rightarrow{}^{b}TM_{2}$, which is the unique morphism of vector bundles ${}^{b}TM_{1}\rightarrow{}^{b}TM_{2}$ that makes the following diagram commute {\cite[Proposition 3.5.2]{RalphThesis}}:

\begin{equation}\label{diag}
\begin{tikzcd}[column sep=large, row sep=large]
{}^{b}TM_{1}\arrow{r}{{}^{b}f_{*}} \arrow{d}{\rho_{1}} & {}^{b}TM_{2} \arrow{d}{\rho_{2}}\\
TM_{1}\arrow{r}{f_{*}}&TM_{2}\\
\end{tikzcd}.
\end{equation}
\squeezeup

At each point $p\in M_{1}$, the derivative $\left(f_{*}\right)_{p}$ and the $b$-derivative $\left({}^{b}f_{*}\right)_{p}$ have the same rank, by the next result proved in \cite{CK}.

\begin{lemma}\label{rank}
Let $f:(M_{1},Z_{1})\rightarrow (M_{2},Z_{2})$ be a $b$-map. The anchor $\rho_{1}$ of ${}^{b}TM_{1}$ restricts to an isomorphism $\left(\rho_{1}\right)_{p}:\text{Ker}\left({}^{b}f_{*}\right)_{p}\rightarrow\text{Ker}\left(f_{*}\right)_{p}$ for all $p\in M_{1}$.
\end{lemma}

We finish this subsection by observing that, if a $b$-vector field {can be pushed forward by} the derivative $f_*$ of a $b$-map $f$, then its lift to a section of the $b$-tangent bundle {can be pushed forward by} the $b$-derivative ${}^{b}f_{*}$.
\begin{lemma}\label{bder}
	Let $f:(M_{1},Z_{1})\rightarrow(M_{2},Z_{2})$ be a {surjective} $b$-map,
	and let $\overline{Y}\in\Gamma({}^{b}TM_{1})$ be such that ${Y:=}\rho_{1}(\overline{Y})$ {pushes forward} to some element $W\in\mathfrak{X}(M_2)$. Then ${}^{b}f_{*}(\overline{Y})$ is a well-defined section of ${}^{b}TM_{2}$, and it equals the unique element $\overline{W}\in\Gamma\left({}^{b}TM_{2}\right)$ satisfying $\rho_{2}(\overline{W})=W$.
\end{lemma}
\begin{proof}
	{Since $f$ is a $b$-map, we have that $W\in\mathfrak{X}(M_2)$ is tangent to $Z_{2}$, so indeed $W=\rho_{2}(\overline{W})$ for unique $\overline{W}\in\Gamma({}^{b}TM_{2})$. Now,
		first consider $p\in M_{1}\setminus Z_{1}$. Commutativity of the diagram \eqref{diag} implies that
		\[
		\rho_{2}\left(\left({}^{b}f_{*}\right)_{p}\left(\overline{Y}_{p}\right)\right)=\left(f_{*}\right)_{p}\left(\rho_{1}\left(\overline{Y}_{p}\right)\right)=\left(f_{*}\right)_{p}(Y_{p})=W_{f(p)}.
		\]
		But we also have $\rho_{2}\left(\overline{W}_{f(p)}\right)=W_{f(p)}$, so that injectivity of $\rho_{2}$ at $f(p)\in M_{2}\setminus Z_{2}$ implies $\left({}^{b}f_{*}\right)_{p}\left(\overline{Y}_{p}\right)=\overline{W}_{f(p)}$. Next, we choose $p\in Z_{1}$. Since $f$ is a $b$-map, we can take a ({one-dimensional}) slice $S$ through $p$ transverse to $Z_{1}$, such that the restriction $\left.f\right|_{S}:S\rightarrow f(S)$ is a diffeomorphism. Since $\left.\left({}^{b}f_{*}\right)\right|_{S}$ is a vector bundle map covering the diffeomorphism $f|_{S}$, the expression $\left.\left({}^{b}f_{*}\right)\right|_{S}\left(\left.\overline{Y}\right|_{S}\right)$ is well-defined and smooth. Moreover, it is equal to $\overline{W}|_{f(S)}$ on the dense subset $f(S)\setminus(f(S)\cap Z_{2})\subset f(S)$, as we just proved. By continuity, the equality $\left.\left({}^{b}f_{*}\right)\right|_{S}\left(\left.\overline{Y}\right|_{S}\right)=\overline{W}|_{f(S)}$ holds on all of $f(S)$, so that in particular $\left({}^{b}f_{*}\right)_{p}\left(\overline{Y}_{p}\right)=\overline{W}_{f(p)}$. This concludes the proof.}\qedhere
		\end{proof}

\subsection{$b$-submanifolds}
\leavevmode
\vspace{0.15cm} 

Given a $b$-manifold $(M,Z)$, a submanifold $C\subset M$ transverse to $Z$ inherits a $b$-manifold structure with distinguished  hypersurface $C\cap Z$. Such submanifolds are therefore the natural subobjects in the $b$-category.

\begin{defi}
	A \emph{$b$-submanifold} $C$ of a $b$-manifold $(M,Z)$ is a submanifold $C\subset M$ that is transverse to $Z$.
\end{defi}

Let $C\subset (M,Z)$ be a $b$-submanifold. The inclusion $i:(C,C\cap Z)\hookrightarrow (M,Z)$ of $b$-manifolds induces a canonical map ${}^{b}i_{*}:{}^{b}TC\rightarrow{}^{b}TM$ that is injective by Lemma \ref{rank}. This allows us to view ${}^{b}TC$ as a Lie subalgebroid of ${}^{b}TM$. In particular, we have the following fact.

\begin{lemma}\label{L}
	If $C\subset (M,Z)$ is a $b$-submanifold, then $\mathbb{L}_{p}\subset{}^{b}T_{p}C$ for all $p\in C\cap Z$.  
\end{lemma}
\begin{proof}
	Fixing some notation, we have anchor maps
	$\tilde{\rho}:{}^{b}TC\rightarrow TC$ and $\rho:{}^{b}TM\rightarrow TM$, and we put $\widetilde{\mathbb{L}}:=\text{Ker}\left(\tilde{\rho}|_{C\cap Z}\right)$ and $\mathbb{L}=\text{Ker}\left(\rho|_{Z}\right)$ as before. If $i:(C,C\cap Z)\hookrightarrow (M,Z)$ denotes the inclusion, then we get a commutative diagram with exact rows, for points $p\in C\cap Z$:
	\begin{equation}\label{diag2}
	\begin{tikzcd}[column sep=large, row sep=large]
	0 \arrow{r}&\mathbb{L}_{p}\arrow[hookrightarrow,r] & {}^{b}T_{p}M\arrow{r}{\rho} & T_{p}Z \arrow{r}& 0\\
	0\arrow{r} &\widetilde{\mathbb{L}}_{p}\arrow[hookrightarrow,r] & {}^{b}T_{p}C\arrow{u}{\left({}^{b}i_{*}\right)_{p}} \arrow{r}{\tilde{\rho}} & T_{p}(C\cap Z)\arrow{u}{\left(\left(i|_{C\cap Z}\right)_{*}\right)_{p}} \arrow{r} &0 \\
	\end{tikzcd}.
	\end{equation}
	\squeezeup

We obtain $\left({}^{b}i_{*}\right)_{p}\left(\widetilde{\mathbb{L}}_{p}\right)=\mathbb{L}_{p}$: the inclusion ``$\subset$'' holds by the above diagram, and the equality follows by dimension reasons since ${\left({}^{b}i_{*}\right)_{p}} $ is injective. In particular, 
$\mathbb{L}_{p}$ is contained in the image of ${\left({}^{b}i_{*}\right)_{p}}$,  as we wanted to show.
\end{proof}

The notions of $b$-map and $b$-submanifold are compatible, as the next lemma shows. %

\begin{lemma}\label{restriction}
Let $f:(M_{1},Z_{1})\rightarrow(M_{2},Z_{2})$ be a $b$-map, and assume that we have $b$-submanifolds $C_{1}\subset(M_{1},Z_{1})$ and $C_{2}\subset(M_{2},Z_{2})$ such that $f(C_{1})\subset C_{2}$. 
	\begin{enumerate}[a)]
		\item  Restricting $f$ gives a $b$-map \[\left.f\right|_{C_{1}}:(C_{1},C_{1}\cap Z_{1})\rightarrow (C_{2},C_{2}\cap Z_{2}).\]
		\item 
		Further, $\left.\left({}^{b}f_{*}\right)\right|_{{}^{b}TC_1}={}^{b}\left(f|_{C_1}\right)_{*}$. 
	\end{enumerate}
\end{lemma}
\begin{proof}
	\begin{enumerate}[a)]
		\item We first note that
		\begin{align*}
		(\left.f\right|_{C_{1}})^{-1}\left(C_{2}\cap Z_{2}\right)&=C_{1}\cap f^{-1}\left(C_{2}\cap Z_{2}\right)=C_{1}\cap f^{-1}\left(C_{2}\right)\cap f^{-1}\left(Z_{2}\right)\\
		&=C_{1}\cap f^{-1}\left(C_{2}\right)\cap Z_{1}=C_{1}\cap Z_{1},
		\end{align*}
		since $f$ is a $b$-map and $C_{1}\subset f^{-1}\left(C_{2}\right)$. Next, choosing $p\in C_{1}\cap Z_{1}$, we have to show that
		\begin{equation}\label{transversality}
		\left(f_{*}\right)_{p}\left(T_{p}C_{1}\right)+T_{f(p)}\left(C_{2}\cap Z_{2}\right)=T_{f(p)}C_{2}.
		\end{equation}
		We clearly have the inclusion ``$\subset$''. For the reverse, we choose $v\in T_{f(p)}C_{2}$. By transversality $f\pitchfork Z_{2}$, we know that $\left(f_{*}\right)_{p}\left(T_{p}M_{1}\right)+ T_{f(p)}Z_{2}=T_{f(p)}M_{2}$. So we have $v=\left(f_{*}\right)_{p}(x)+y$ for some $x\in T_{p}M_{1}$ and $y\in T_{f(p)}Z_{2}$. Next, since $C_{1}\pitchfork Z_{1}$, we have $T_{p}C_{1}+T_{p}Z_{1}=T_{p}M_{1}$ so that $x=x_{1}+x_{2}$ for some $x_{1}\in T_{p}C_{1}$ and $x_{2}\in T_{p}Z_{1}$. So we have
		\begin{equation}\label{sum}
		v=\left(f_{*}\right)_{p}(x_{1})+\left[\left(f_{*}\right)_{p}(x_{2})+y\right].
		\end{equation}
		The term in square brackets clearly lies in $T_{f(p)}Z_{2}$, and being equal to 
			$v-\left(f_{*}\right)_{p}(x_{1})$ it also lies in $T_{f(p)}C_{2}$.  So it lies
			in $T_{f(p)}\left(C_{2}\cap Z_{2}\right)$, using the transversality $C_2\pitchfork Z_{2}$.
		Hence the decomposition \eqref{sum} is as required in \eqref{transversality}.
		\item Denoting the inclusions $i_{1}:(C_1,C_1\cap Z_{1})\hookrightarrow (M_{1},Z_{1})$ and $i_{2}:(C_2,C_2\cap Z_{2})\hookrightarrow (M_{2},Z_{2})$, we have $f\circ i_{1}=i_{2}\circ f|_{C_1}$. Hence by functoriality, ${}^{b}f_{*}\circ{}^{b}(i_{1})_{*}={}^{b}(i_{2})_{*}\circ{}^{b}\left(f|_{C_1}\right)_{*}$, which implies the claim. 
	\end{enumerate}
\end{proof}

\subsection{Distributions on $b$-manifolds}
\leavevmode
\vspace{0.15cm}

We saw that the short exact sequence \eqref{eq:sesL} does not split canonically. However, its restriction to suitable distributions does split.

\begin{lemma}\label{lem:distrsplit}
Let $(M,Z)$ be a $b$-manifold with anchor map $\rho:{}^{b}TM\rightarrow TM$.
\begin{enumerate}[a)]
\item Given a distribution $D$ on $M$ that is tangent to $Z$,  there exists a canonical splitting $\sigma\colon D\rightarrow{}^{b}TM$ of the anchor $\rho$. 

\item Let $\mathcal{D}$ denote the set of distributions on $M$ tangent to $Z$, and let $\mathcal{S}$ consist of the subbundles of ${}^{b}TM$ intersecting trivially $\ker(\rho)$. Then there is a bijection
\[
\mathcal{D}\rightarrow\mathcal{S}:D\mapsto\sigma(D),
\]
where the splitting $\sigma$ is as in $a)$. The inverse map reads $D'\mapsto\rho(D')$.
\end{enumerate}
\end{lemma}

\begin{proof}
a) One checks that the inclusion $\Gamma(D)\subset\Gamma\left({}^{b}TM\right)$ induces a well-defined vector bundle map
\[
\sigma:D\rightarrow{}^{b}TM:v\mapsto X_{p},
\]
where $X\in\Gamma(D)$ is any extension of $v\in D_{p}$. This map $\sigma$ satisfies  $\rho\circ\sigma=\text{Id}_{D}$, so in particular $\rho(\sigma(D))=D$.

b) We only have to show that if $D'$ is a subbundle of ${}^{b}TM$ intersecting trivially $\ker(\rho)$, then $\sigma(\rho(D'))=D'$. Denote $D:=\rho(D')$, a distribution  on $M$ tangent to $Z$.
The canonical splitting $\sigma\colon D\rightarrow{}^{b}TM$ is injective, and $D$ and $D'$ have the same rank,
hence it suffices to show that $\sigma(D)\subset D'$. If $X$ is a section of $D$, then $X=\rho(Y)$ for unique $Y\in\Gamma\left(D'\right)$. We get
\[
\rho(\sigma(X))=X=\rho(Y),
\]
and since the anchor $\rho$ is injective on sections, this implies that $\sigma(X)=Y$.
\end{proof}

\begin{cor}\label{split}
Let $f:(M_{1},Z_{1})\rightarrow (M_{2},Z_{2})$ be a $b$-map of constant rank.  Notice that $\text{Ker}(f_{*})$ is a distribution on $M_{1}$ that is tangent to $Z_{1}$.  
It satisfies $$\sigma\left(\text{Ker}(f_{*})\right)=\text{Ker}\left({}^{b}f_{*}\right),$$ where $\sigma:\text{Ker}(f_{*})\rightarrow {}^{b}TM_{1}$ denotes the canonical splitting of the anchor $\rho_{1}$.
\end{cor}
\begin{proof}
Under the bijection of Lemma \ref{lem:distrsplit} b), $\text{Ker}(f_{*})$ corresponds to  $\text{Ker}\left({}^{b}f_{*}\right)$, as a consequence of {Lemma \ref{rank}.}
\end{proof}

\subsection{Vector bundles in the $b$-category}
\leavevmode
\vspace{0.15cm}

If $(M,Z)$ is a $b$-manifold and $\pi:E\rightarrow M$ a vector bundle, then $(E,E|_{Z})$ is naturally a $b$-manifold and the projection $\pi:(E,E|_{Z})\rightarrow (M,Z)$ is a $b$-map. Along the zero section $M\subset E$, the $b$-tangent bundle ${}^{b}TE$ splits canonically as follows.

\begin{lemma}\label{decomposition}
	Let $(M,Z)$ be a $b$-manifold and $\pi:E\rightarrow M$ a vector bundle. Then at points $p\in M$ we have a canonical decomposition
	\[
	{}^{b}T_{p}E\cong{}^{b}T_{p}M\oplus E_{p}.
	\]
\end{lemma}

\begin{proof}
		Denote by $VE:=\text{Ker}(\pi_{*})$ the vertical bundle.
	By Corollary \ref{split}
	there is a canonical lift $\sigma:VE\hookrightarrow{}^{b}TE$ such that $\sigma(VE)=\text{Ker}({}^{b}\pi_{*})$. So we get a short exact sequence of vector bundles over $E$
	\begin{equation}\label{eq:ses}
	0 \longrightarrow VE \xhookrightarrow{\sigma}{}^{b}TE
	\overset{\widetilde{{}^{b}\pi_{*}}}{\hspace{0.25cm}\longrightarrow}\pi^{*}\left({}^{b}TM\right)\rightarrow 0.
	\end{equation}
	Here
		\begin{equation*}
		\pi^{*}\left({}^{b}TM\right)=\left\{(e,v)\in E\times{}^{b}TM:\pi(e)=pr(v)\right\}
		\end{equation*}
		is the pullback of the vector bundle $pr:{}^{b}TM\rightarrow M$ by $\pi$, and	the surjective vector bundle map
				\[
		\widetilde{{}^{b}\pi_{*}}:{}^{b}TE\rightarrow \pi^{*}\left({}^{b}TM\right),\;(e,v)\mapsto \left(e,\left({}^{b}\pi_{*}\right)_{e}(v)\right)
		\]
		is induced by the $b$-map $\pi:\left(E,E|_{Z}\right)\rightarrow (M,Z)$.

	Restricting \eqref{eq:ses} to the zero section $M\subset E$ gives a short exact sequence of vector bundles over $M$: 
	\[
	0\longrightarrow E \hookrightarrow{}^{b}TE|_{M}\overset{{{}^{b}\pi_{*}}}{\hspace{0.25cm}\longrightarrow}{}^{b}TM\rightarrow 0.
	\]
	This sequence splits canonically through the map ${}^{b}i_{*}:{}^{b}TM\rightarrow {}^{b}TE|_{M}$ induced by the inclusion $i:(M,Z)\hookrightarrow (E,E|_{Z})$. 
\end{proof}

The following result makes use of the decomposition introduced in Lemma
\ref{decomposition}.

\begin{lemma}\label{maps}
	\begin{enumerate}[a)]
		\item Let $\pi:(E,E|_{Z})\rightarrow (M,Z)$ be a vector bundle over the $b$-manifold $(M,Z)$. Denote by $\rho$ and $\widetilde{\rho}$ the anchor maps of ${}^{b}TM$ and ${}^{b}TE$ respectively. Under the decomposition of Lemma \ref{decomposition}, we have that the map
		\[
		\widetilde{\rho}|_{M}:{}^{b}TE|_{M}\cong{}^{b}TM\oplus E\longrightarrow TE|_{M}\cong TM\oplus E
		\]
		equals $\rho\oplus\text{Id}_{E}$.
		
		\item Consider a morphism $\varphi$ of vector bundles over $b$-manifolds covering a $b$-map $f$:
		\begin{equation}\label{eq:VB}
		\begin{tikzcd}[column sep=large, row sep=large]
		(E_{1},E_{1}|_{Z_{1}}) \arrow{d}{\pi_1} \arrow{r}{\varphi}& (E_{2},E_{2}|_{Z_{2}})\arrow{d}{\pi_2}\\
		(M_{1},Z_{1})\arrow{r}{f}&(M_{2},Z_{2})\\
		\end{tikzcd}.
		\squeezeup
		\end{equation}	
		{Then $\varphi$ is a $b$-map,} and its $b$-derivative along the zero section $$ {}^{b}\varphi_{*}|_M\colon
			{}^{b}TE_{1}|_M\cong{}^{b}TM_{1}\oplus E_{1}\to {}^{b}TE_{2}|_M\cong{}^{b}TM_{2}\oplus E_{2}	 $$
			equals ${}^{b}f_{*}\oplus \varphi$.
	\end{enumerate}
\end{lemma}
\begin{proof}
	\begin{enumerate}[a)]
		\item Since $M$ is a $b$-submanifold of $(E,E|_{Z})$, we have that ${}^{b}TM$ is a Lie subalgebroid of ${}^{b}TE$. In particular, $\widetilde{\rho}$ and $\rho$ agree on ${}^{b}TM$. Next, we know that $\widetilde{\rho}$ takes $E\subset{}^{b}TE|_{M}$ isomorphically to $E\subset TE|_{M}$, thanks to Lemma \ref{rank} applied to $\pi$. To see that $\widetilde{\rho}|_{E}=\text{Id}_{E}$, we choose $v\in E_{p}$ and extend it to $V\in\Gamma(VE)$. Denote by $\sigma:VE\hookrightarrow{}^{b}TE$ the canonical splitting of $\widetilde{\rho}$, as in the proof of Lemma \ref{decomposition}. Then $\widetilde{\rho}(v)=\left[\widetilde{\rho}(\sigma(V))\right]_{p}=V_{p}=v$.
		\item {It is routine to check that $\varphi$ is a $b$-map, so we only prove the second statement.}
		Taking the $b$-derivative of both sides of the equality $\pi_2\circ \varphi=f\circ \pi_1$ at a point $p\in M_{1}$, we know that $\left({}^{b}\pi_{2}\right)_{*}\left({}^{b}\varphi_{*}(E_{1})_{p}\right)={}^{b}f_{*}\left(\left({}^{b}\pi_{1}\right)_{*}(E_{1})_{p}\right)=0$, since $\left(E_{1}\right)_{p}=\text{Ker}\left[\left({}^{b}\pi_{1}\right)_{*}\right]_{p}$. Hence ${}^{b}\varphi_{*}(E_{1})_{p}\subset\text{Ker}\left[\left({}^{b}\pi_{2}\right)_{*}\right]_{f(p)}=\left(E_{2}\right)_{f(p)}$ by the proof of Lemma \ref{decomposition}. 
		Using  a) and the diagram \eqref{diag}, we have a commutative diagram 
		\begin{equation}\label{diag3}
		\begin{tikzcd}[column sep=large, row sep=large]
		{}^{b}T_{p}E_{1}\cong{}^{b}T_{p}M_{1}\oplus\left(E_{1}\right)_{p}\arrow{r}{^{b}\varphi_{*}} \arrow{d}{\left(\rho_{1}\oplus\text{Id}\right)} &  {}^{b}T_{f(p)}E_{2}\cong{}^{b}T_{f(p)}M_{2}\oplus\left(E_{2}\right)_{f(p)}\arrow{d}{(\rho_{2}\oplus\text{Id})}\\
		T_{p}E_{1}\cong T_{p}M_{1}\oplus(E_{1})_{p}\arrow{r}{\varphi_{*}}&T_{f(p)}E_{2}\cong T_{f(p)}M_{2}\oplus\left(E_{2}\right)_{f(p)}\\
		\end{tikzcd}.
		\squeezeup
		\end{equation}
		It implies that
		\[
		\left.{}^{b}\varphi_{*}\right|_{(E_{1})_{p}}=\left.\varphi_{*}\right|_{(E_{1})_{p}}=\left.\varphi\right|_{(E_{1})_{p}}.
		\]
		Finally,  $ {}^{b}\varphi_{*}|_{{}^{b}TM_{1}}={}^{b}f_{*}$ holds by Lemma \ref{restriction} b).	\qedhere		
	\end{enumerate}
\end{proof}

\subsection{Log-symplectic and $b$-symplectic structures}\label{subsec:logb}
\leavevmode
\vspace{0.15cm}

The $b$-geometry formalism can be used to describe a certain class of Poisson structures, called log-symplectic structures. These can indeed be regarded as symplectic structures on the $b$-tangent bundle.

\begin{defi}
A \emph{Poisson structure} on a manifold $M$ is a bivector field $\Pi\in\Gamma\left(\wedge^{2}TM\right)$ such that the bracket $\{f,g\}=\Pi(df,dg)$ is a Lie bracket on $C^{\infty}(M)$. Equivalently, the bivector field $\Pi$ must satisfy $[\Pi,\Pi]=0$, where $[\cdot,\cdot]$ is the Schouten-Nijenhuis bracket of multivector fields. 
A smooth map $f:\left(M_{1},\Pi_{1}\right)\rightarrow\left(M_{2},\Pi_{2}\right)$ is a \emph{Poisson map} if the pullback $f^{*}:\left(C^{\infty}(M_{2}),\{\cdot,\cdot\}_{2}\right)\rightarrow\left(C^{\infty}(M_{1}),\{\cdot,\cdot\}_{1}\right)$ is a Lie algebra homomorphism.
\end{defi}

The bivector $\Pi$ induces a bundle map $\Pi^{\sharp}:T^{*}M\rightarrow TM$ by
\[
\left\langle \Pi_{p}^{\sharp}(\alpha),\beta\right\rangle=\Pi_{p}(\alpha,\beta)\hspace{1cm}\forall\alpha,\beta\in T_{p}^{*}M,
\] 
and the \emph{rank} of $\Pi$ at $p\in M$ is defined to be the rank of the linear map $\Pi^{\sharp}_{p}$. Poisson structures of full rank correspond with symplectic structures via $\omega\leftrightarrow-\Pi^{-1}$.

For every $f\in C^{\infty}(M)$, the operator $\{f,\cdot\}$ is a derivation of $C^{\infty}(M)$. The corresponding vector field $X_{f}=\Pi^{\sharp}(df)$ is the \emph{Hamiltonian vector field} of $f$. Any Poisson manifold $(M,\Pi)$ comes with a (singular) distribution $\text{Im}\left(\Pi^{\sharp}\right)$, generated by the Hamiltonian vector fields. This distribution is integrable (in the sense of Stefan-Sussman) and each leaf $\mathcal{O}$ of the associated foliation has an induced symplectic structure $\omega_{\mathcal{O}}:=-\left(\Pi|_{\mathcal{O}}\right)^{-1}$.

\begin{defi}
A Poisson structure $\Pi$ on a manifold $M^{2n}$ is called \emph{log-symplectic} if $\wedge^{n}\Pi$ is transverse to the zero section of the line bundle $\wedge^{2n}TM$.  
\end{defi}

Note that a log-symplectic structure $\Pi$ is of full rank everywhere, except at points lying in the set $Z:=\left(\wedge^{n}\Pi\right)^{-1}(0)$, called the singular locus of $\Pi$.  If $Z$ is nonempty, then it is a smooth hypersurface by the transversality condition, and we call $\Pi$ bona fide log-symplectic. In that case, $Z$ is a Poisson submanifold of $(M,\Pi)$ with an induced Poisson structure that is regular of corank-one. If $Z$ is empty, then $\Pi$ defines a symplectic structure on $M$.

Since log-symplectic structures come with a specified hypersurface, it seems plausible that they have a $b$-geometric interpretation. As it turns out, log-symplectic structures are exactly the symplectic structures of the $b$-category.

\begin{defi}
	A \emph{$b$-symplectic form} on a $b$-manifold $(M^{2n},Z)$ is a ${}^{b}d$-closed and non-degenerate $b$-two-form $\omega\in{}^{b}\Omega^{2}(M)$.
\end{defi}
Here, non-degeneracy means that the bundle map $\omega^{\flat}:{}^{b}TM\rightarrow{}^{b}T^{*}M$ is an isomorphism, or equivalently that $\wedge^{n}\omega$ is a nowhere vanishing element of ${}^{b}\Omega^{2n}(M)$.

\begin{ex}\label{ex:bcotg} \cite[Example 9]{miranda}
	In analogy with the symplectic case, the $b$-cotangent bundle ${}^{b}T^{*}M$ of a $b$-manifold $(M,Z)$ is $b$-symplectic in a canonical way. Note that $\left({}^{b}T^{*}M,\left.{}^{b}T^{*}M\right|_{Z}\right)$ is naturally a $b$-manifold, and that the bundle projection $\pi:\left({}^{b}T^{*}M,\left.{}^{b}T^{*}M\right|_{Z}\right)\rightarrow (M,Z)$ is a $b$-map. The tautological $b$-one-form $\theta\in{}^{b}\Omega^{1}\left({}^{b}T^{*}M\right)$ is defined by
	\[
	\theta_{\xi}(v)=\left\langle \xi,\left({}^{b}\pi_{*}\right)_{\xi}(v)\right\rangle,
	\]
	where $\xi\in{}^{b}T^{*}_{\pi(\xi)}M$ and $v\in {}^{b}T_{\xi}\left({}^{b}T^{*}M\right)$. Its differential $-{}^{b}d\theta$ is a $b$-symplectic form on ${}^{b}T^{*}M$. To see this, choose coordinates $(x_{1},\ldots,x_{n})$ on $M$ adapted to $Z=\{x_{1}=0\}$, and let $(y_{1},\ldots,y_{n})$ denote the fiber coordinates on ${}^{b}T^{*}M$ with respect to the local frame $\left\{\frac{dx_{1}}{x_{1}},dx_{2},\ldots,dx_{n}\right\}$. The tautological $b$-one form is then given by
	\[
	\theta=y_{1}\frac{dx_{1}}{x_{1}}+\sum_{i=2}^{n}y_{i}dx_{i},
	\]
	with exterior derivative 
	\[
	-{}^{b}d\theta=\frac{dx_{1}}{x_{1}}\wedge dy_{1}+\sum_{i=2}^{n}dx_{i}\wedge dy_{i}.
	\]
\end{ex}

A log-symplectic structure on $M$ with singular locus $Z$ is nothing else but a $b$-symplectic structure on the $b$-manifold $(M,Z)$, see \cite[Proposition 20]{miranda}. Indeed, given a $b$-symplectic form $\omega$ on $(M,Z)$,
its negative inverse ${}^{b}\Pi^{\sharp}:={-}\left(\omega^{\flat}\right)^{-1}:{}^{b}T^{*}M\rightarrow{}^{b}TM$  defines a $b$-bivector field ${}^{b}\Pi\in\Gamma\left(\wedge^{2}\left({}^{b}TM\right)\right)$, and applying the anchor map $\rho$ to it yields a bivector field $\Pi:=\rho\left({}^{b}\Pi\right)\in\Gamma\left(\wedge^{2}TM\right)$ that is log-symplectic with singular locus $Z$. Conversely, a log-symplectic structure $\Pi$ on $M$ with singular locus $Z$ lifts uniquely under $\rho$ to a non-degenerate $b$-bivector field ${}^{b}\Pi$, whose negative inverse is a $b$-symplectic form on $(M,Z)$. These processes are summarized in the following diagram:

\begin{equation}\label{lift}
\begin{tikzcd}[column sep=large, row sep=large]
{}^{b}T^{*}M\arrow[r,"{}^{b}\Pi^{\sharp}",swap]  & {}^{b}TM \arrow{d}{\rho}\arrow[l, shift right=1ex,"{-}\omega^{\flat}",swap]\\
T^{*}M\arrow{r}{\Pi^{\sharp}}\arrow{u}{\rho^{*}}&TM\\
\end{tikzcd}.
\end{equation}
\squeezeup

We will switch between the $b$-symplectic and the log-symplectic (i.e. Poisson) viewpoint, depending on which one is the most convenient.

\subsection{A relative $b$-Moser theorem}
\leavevmode
\vspace{0.15cm}

We will need a relative Moser theorem in the $b$-symplectic setting. First, we prove the following $b$-geometric version of the relative Poincar\'e lemma {\cite[Proposition 6.8]{ana}}.

\begin{lemma}\label{poincare}
Let $(M,Z)$ be a $b$-manifold and $C\subset (M,Z)$ a $b$-submanifold. Denote by $i:(C,C\cap Z)\hookrightarrow (M,Z)$ the inclusion. If $\beta\in{}^{b}\Omega^{k}(M)$ is ${}^{b}d$-closed and ${}^{b}i^{*}\beta=0$, then there exist a neighborhood $U$ of $C$ and $\eta\in{}^{b}\Omega^{k-1}(U)$ such that
\[
\begin{cases}
{}^{b}d\eta=\beta|_{U}\\
\eta|_{C}=0
\end{cases}.
\]
\end{lemma}

\begin{proof} We adapt the proof of {\cite[Proposition 6.8]{ana}}.
We first choose a suitable tubular neighborhood of $C$ that is compatible with the hypersurface $Z$. Due to transversality $C\pitchfork Z$, we can pick a complement $V$ to $TC$ in $TM|_{C}$ such that $V_{p}\subset T_{p}Z$ for all $p\in C\cap Z$. Fix a Riemannian metric $g$ for which $Z\subset (M,g)$ is totally geodesic (e.g. \cite[Lemma 6.8]{Milnor}).
The associated exponential map then establishes a $b$-diffeomorphism between a neighborhood of $C$ in $(V,V|_{C\cap Z})$ and a neighborhood of $C$ in $(M,Z)$.

So we may work instead on the total space of $\pi:(V,V|_{C\cap Z})\rightarrow (C,C\cap Z)$. Consider the retraction of $V$ onto $C$ given by ${r}:V\times[0,1]\rightarrow V:(p,v,t)\mapsto (p,tv)$, and notice that the ${r}_{t}$ are $b$-maps. The associated time-dependent vector field $X_{t}$ is given by $X_{t}(p,v)={\frac{1}{t}v}$, which is a $b$-vector field that vanishes along $C$. It follows that we get a well-defined $b$-de Rham homotopy operator
\[
I:{}^{b}\Omega^{k}(V)\rightarrow{}^{b}\Omega^{k-1}(V):\alpha\mapsto\int_{0}^{1}{}^{b}{r}_{t}^{*}(\iota_{X_{t}}\alpha)dt,
\]
which satisfies
\begin{equation}\label{homotopy}
{}^{b}{r}_{1}^{*}\alpha-{}^{b}{r}_{0}^{*}\alpha={}^{b}d I(\alpha)+I({}^{b}d\alpha).
\end{equation}
Since ${r}_{1}=\text{Id}$ and ${r}_{0}=i\circ\pi$, the formula \eqref{homotopy} gives $\beta={}^{b}d I(\beta)$. Now set $\eta:=I(\beta)$.
\end{proof}

\begin{prop}[Relative $b$-Moser theorem]
\label{moser}
Let $(M,Z)$ be a $b$-manifold and $C\subset(M,Z)$ a $b$-submanifold. If $\omega_{0}$ and $\omega_{1}$ are $b$-symplectic forms on $(M,Z)$ such that $\omega_{0}|_{C}=\omega_{1}|_{C}$, then there exists a $b$-diffeomorphism $\varphi$ between neighborhoods of $C$ such that
$\varphi|_{C}=\text{Id}$ and ${}^{b}\varphi^{*}\omega_{1}=\omega_{0}$.
\end{prop}
\begin{proof}
Consider the convex combination $\omega_{t}:=\omega_{0}+t(\omega_{1}-\omega_{0})$ for $t\in[0,1]$. There exists a neighborhood $U$ of $C$ such that $\omega_{t}$ is non-degenerate on $U$ for all $t\in[0,1]$. Shrinking $U$ if necessary, Lemma \ref{poincare} yields $\eta\in{}^{b}\Omega^{1}(U)$ such that $\omega_{1}-\omega_{0}={}^{b}d\eta$ and $\eta|_{C}=0$. As in the usual Moser trick, it now suffices to solve the equation
\[
\iota_{X_{t}}\omega_{t}+\eta=0
\]
for $X_{t}\in{}^{b}\mathfrak{X}(U)$, which is possible by non-degeneracy of $\omega_{t}$. The $b$-vector fields $X_{t}$ thus obtained vanish along $C$ since $\eta|_{C}=0$. Further shrinking $U$ if necessary, we can integrate the $X_{t}$ to an isotopy $\{{\phi_{t}}\}_{t\in[0,1]}$ defined on $U$. Note that the ${\phi_{t}}$ are $b$-diffeomorphisms that restrict to the identity on $C$. By the usual Moser argument, we have ${}^{b}{\phi_{1}}^{*}\omega_{1}=\omega_{0}$, so setting $\varphi:={\phi_{1}}$ finishes the proof.
\end{proof}

\begin{remark}
{We learnt from Ralph Klaasse that the work in progress \cite{KLAlmRegPois} contains a version of Proposition \ref{moser} that holds in the more general setting of symplectic Lie algebroids.}
\end{remark}

\color{black}
\section{$b$-coisotropic submanifolds {and the $b$-Gotay theorem}}\label{sec:bcoiso}

This section is devoted to coisotropic submanifolds of $b$-symplectic manifolds that are transverse to the degeneracy hypersurface. The main result is Theorem \ref{gotay}, a $b$-symplectic  version of Gotay's theorem, which implies a normal form statement around such submanifolds. {This can be used, for instance, to study the deformation theory of $b$-coisotropic submanifolds \cite{GZbdefs}.}

\subsection{$b$-coisotropic submanifolds}\label{bcoiso}
\leavevmode
\vspace{0.15cm}

{In this subsection we introduce $b$-coisotropic submanifolds and we discuss some of their main features.}
First recall the definition of a coisotropic submanifold in Poisson geometry. 

\begin{defi}\label{def:coiso}
Let $(M,\Pi)$ be a Poisson manifold with associated Poisson bracket $\{\cdot,\cdot\}$. A submanifold $C\subset M$ is \emph{coisotropic} if the following equivalent conditions hold:
\begin{enumerate}[a)]
	\item $\Pi^{\sharp}\left(TC^{0}\right)\subset TC$, where $TC^{0}\subset T^{*}M|_{C}$ denotes the annihilator of $TC$.
	\item $\{\mathcal{I}_{C},\mathcal{I}_{C}\}\subset \mathcal{I}_{C}$, where $\mathcal{I}_{C}:=\{f\in C^{\infty}(M): f|_{C}=0\}$ denotes the vanishing ideal of $C$. 
	\item $T_{p}C\cap T_{p}\mathcal{O}$ is a coisotropic subspace of the symplectic vector space $\left(T_{p}\mathcal{O},-\left(\Pi|_{\mathcal{O}}\right)^{-1}_{p}\right)$ for all $p\in C$, where $\mathcal{O}$ denotes the symplectic leaf through $p$.
\end{enumerate}
\end{defi}
The singular distribution $\Pi^{\sharp}\left(TC^{0}\right)$ on $C$ appearing above is called the \emph{characteristic distribution}.
If $\Pi={-}\omega^{-1}$ is symplectic, the coisotropicity condition becomes $TC^{\omega}\subset TC$.

\begin{defi}\label{def:bcoiso}
Let $(M,Z,\omega)$ be a $b$-symplectic manifold, and denote by $\Pi$ the corresponding Poisson bivector field on $M$.
A submanifold $C$ of $M$ is called \emph{$b$-coisotropic} if it is coisotropic with respect to $\Pi$ and {a $b$-submanifold (i.e. transverse to $Z$)}.
\end{defi}

\begin{remark}
A $b$-coisotropic submanifold $C^{n}\subset (M^{2n},Z,\Pi)$ of middle dimension is necessarily Lagrangian, i.e. $T_{p}C\cap T_{p}\mathcal{O}$ is a Lagrangian subspace of the symplectic vector space $\left(T_{p}\mathcal{O},-\left(\Pi|_{\mathcal{O}}\right)^{-1}_{p}\right)$ for all $p\in C$, where $\mathcal{O}$ denotes the symplectic leaf through $p$. {Indeed, at points away from $Z$ there is nothing to prove. At points $p\in C\cap Z$, we have \[
\dim\left(T_{p}C\cap T_{p}\mathcal{O}\right)\leq \dim\left(T_{p}C\cap T_{p}Z\right)=n-1,
\]
where the last equality follows from transversality $C\pitchfork Z$. On the other hand, $T_{p}C\cap T_{p}\mathcal{O}$ is at least $(n-1)$-dimensional, being a coisotropic subspace of the $(2n-2)$-dimensional symplectic vector space $T_{p}\mathcal{O}$. Hence $\dim(T_{p}C\cap T_{p}\mathcal{O})=n-1$, which proves the claim.
}
\end{remark}

Definition \ref{def:bcoiso} can be rephrased in terms of the $b$-symplectic form $\omega$:
a $b$-coisotropic submanifold is precisely a 
$b$-submanifold $C$ such that $({}^{b}TC)^\omega\subset{}^{b}TC$. 
\begin{prop}\label{distribution}
Let $C$ be a $b$-submanifold of a $b$-symplectic manifold $(M,Z,\omega)$. Then $C$ is  coisotropic if and only if
$({}^{b}TC)^\omega\subset{}^{b}TC$.
\end{prop}
Notice that the latter condition states that $^bTC$ is a coisotropic subbundle of the symplectic vector bundle $(^bTM|_C, \omega|_{C})$.

\begin{proof}
If $C$ is coisotropic, then at points of $C \cap (M\setminus Z)$ we have that
$TC^\omega\subset TC$, i.e. $({}^{b}TC)^\omega\subset{}^{b}TC$. By continuity, this inclusion of subbundles holds at all points of $C$. Conversely, if this inclusion holds on $C$, it follows that $C \cap (M\setminus Z)$ is coisotropic in $M\setminus Z$, and using characterization b) in Definition \ref{def:coiso} we see that $C$ is coisotropic in $M$.
\end{proof}

We give an alternative description of the characteristic distribution of a $b$-coisotropic submanifold.

\begin{lemma}\label{lem:twochar}
Let $C$ be any $b$-submanifold of a $b$-symplectic manifold $(M,Z,\omega)$, {and let $\rho:{}^{b}TM\rightarrow TM$ denote the anchor of ${}^{b}TM$ so that $\Pi=\rho\left({-}\omega^{-1}\right)$ is the Poisson bivector corresponding with $\omega$.} Then
\begin{equation}\label{distr}
\rho\left(\left({}^{b}TC\right)^{\omega}\right)=\Pi^{\sharp}\left(TC^{0}\right).
\end{equation}
\end{lemma}
\begin{proof}
At points $p\in C\setminus(C\cap Z)$, the equality \eqref{distr} holds by symplectic linear algebra. So let $p\in C\cap Z$. 
Denote by ${}^{b}\Pi:={-}\omega^{-1}\in\Gamma\left(\wedge^{2}\left({}^{b}TM\right)\right)$ the lift of $\Pi$ as a $b$-bivector field. Note that
\begin{equation}\label{note}
\left({}^{b}T_{p}C\right)^{\omega_{p}}=\left(\omega_{p}^{\flat}\right)^{-1}\left(\left({}^{b}T_{p}C\right)^{0}\right)={}^{b}\Pi^{\sharp}\left(\left({}^{b}T_{p}C\right)^{0}\right),
\end{equation}
where the annihilator is taken in ${}^{b}T^{*}_{p}M$. We now assert:

\begin{claim*}
$\left({}^{b}T_{p}C\right)^{0}=\rho_{p}^{*}\left(T_{p} C^{0}\right).$
\end{claim*}

\noindent
To prove the claim, we first note that the dimensions of both sides agree since
\[
\text{Ker}(\rho_{p}^{*})\cap T_{p}C^{0}=\text{Im}(\rho_{p})^{0}\cap T_{p}C^{0}=T_{p}Z^{0}\cap T_{p}C^{0}=(T_{p}Z+T_{p}C)^{0}=\{0\},
\]
where the last equality holds by transversality $C\pitchfork Z$. Now it is enough to show that the inclusion ``$\supset$'' holds, which is clearly the case since $\rho_{p}({}^{b}T_{p}C)\subset T_{p}C$.
\hfill $\bigtriangleup$

\vspace{0.3cm}
\noindent
We thus obtain
\[
\rho_{p}\left(\left({}^{b}T_{p}C\right)^{\omega_{p}}\right)=\left(\rho_{p}\circ{}^{b}\Pi_{p}^{\sharp}\circ\rho_{p}^{*}\right)\left(T_{p}C^{0}\right)
=\Pi_{p}^{\sharp}\left(T_{p}C^{0}\right),
\]
where in the first equality we used \eqref{note} and the claim just proved, and in the second we used the diagram \eqref{lift}.
\end{proof}

A general fact  in Poisson geometry is that the conormal bundle of any coisotropic submanifold is a Lie subalgebroid of the cotangent Lie algebroid. We now show that 
the $b$-geometry version of this fact holds for
 $b$-coisotropic submanifolds.

\begin{prop} 
Let $(M,Z,\omega)$ be a $b$-symplectic manifold with corresponding Poisson bivector field $\Pi$. Recall that ${}^{b}T^{*}M$ is a Lie algebroid (endowed with the Lie bracket induced by  ${}^{b}\Pi$), fitting in the diagram of Lie algebroids \eqref{lift}. Let $C$ be a $b$-coisotropic submanifold.
\begin{itemize}
\item[a)]   $({}^{b}TC)^{\circ}$ is a Lie subalgebroid of   ${}^{b}T^{*}M$.
\item[b)] {$({}^{b}TC)^{\circ}$} fits in the following diagram of Lie subalgebroids of the diagram  \eqref{lift}: 
\begin{equation}\label{liftC}
\begin{tikzcd}[column sep=large, row sep=large]
({}^{b}TC)^{\circ}\arrow[r,"{}^{b}\Pi^{\sharp}",swap,"\simeq"',yshift = -2pt] & ({}^{b}TC)^{\omega} \arrow{d}{\rho}\arrow[l, shift right=1ex,"{-}\omega^{\flat}",swap,yshift=2pt]\\
TC^{\circ}\arrow{r}{\Pi^{\sharp}}\arrow[u,"\rho^{*}", "\simeq"'{rotate=90,yshift=-5pt,xshift=-6pt}]&TC
\end{tikzcd}.
\end{equation}
\end{itemize}
\end{prop}

\begin{proof}
Diagram \eqref{liftC} is a diagram of vector subbundles of diagram \eqref{lift}, by the claim in the proof of Lemma \ref{lem:twochar} and by equation \eqref{note}.

For a), since the   morphism ${}^{b}\Pi^{\sharp}$ in diagram \eqref{lift} is an isomorphism of Lie algebroids, it suffices to show that $({}^{b}TC)^{\omega}$ is a Lie subalgebroid of ${}^{b}TM$. Since $({}^{b}TC)^{\omega}$ is the kernel of the closed $b$-2-form ${}^{b}i^* \omega$, a standard Cartan calculus computation shows that this is indeed the case. It is well-known that $TC^{\circ}$ and $TC$ are also Lie subalgebroids, proving b).
\end{proof}

\subsection{Examples of $b$-coisotropic submanifolds}
\leavevmode
\vspace{0.15cm}

{We now exhibit some examples of $b$-coisotropic submanifolds. The main result of this subsection is Proposition \ref{prop:coisoBl}, which shows that graphs of suitable Poisson maps between log-symplectic manifolds give rise to $b$-coisotropic submanifolds, once lifted to a certain blow-up.}

\begin{exs}\label{ex:sympllog}
\begin{enumerate}[a)]
	\item Given a log-symplectic manifold $(M,Z,\Pi)$, any hypersurface  of $M$ transverse to $Z$ is $b$-coisotropic.
	\item Let $(M,\Omega)$ be a symplectic manifold, whose non-degenerate Poisson structure we denote $\Pi_M:=-\Omega^{-1}$, and let $(N,\Pi_{N})$ be a log-symplectic manifold with singular locus $Z$. Then $(M\times N, \Pi_M-\Pi_{N})$ is log-symplectic with singular locus $M\times Z$. Given a Poisson map $\phi:(M, \Pi_M)\rightarrow (N,\Pi_{N})$ transverse to $Z$, we have that $\text{Graph}(\phi)\subset(M\times N,  \Pi_M-\Pi_{N})$ is $b$-coisotropic. {As a concrete example, consider for instance 
\[
\phi:\left(\mathbb{R}^{4},\sum_{i=1}^{2}\pd{x_{i}}\wedge \pd{y_{i}}\right)\rightarrow\left(\mathbb{R}^{2},x\partial_{x}\wedge\partial_{y}\right):
,\;(x_{1},y_{1},x_{2},y_{2})\mapsto (y_{1},x_{2}-x_{1}y_{1}).
\]	
}
\end{enumerate}
\end{exs}


We will now prove Proposition \ref{prop:coisoBl}. We start recalling some facts from  \cite[\S 2.1]{GuLi}. Given a manifold $M$ and a closed submanifold $L$ of codimension $\ge 2$, one can construct a new manifold by replacing $L$ with the projectivization of its normal bundle. The resulting manifold $Bl_L(M)$,
 the \emph{real projective blow-up} of $M$ along $L$,  comes with a map $$p\colon Bl_L(M)\to M$$ which restricts to a diffeomorphism $Bl_L(M)\setminus p^{-1}(L)\to M\setminus L$. Further, let $S\subset M$ be a submanifold which intersects cleanly $L$, i.e. $S\cap L$ is a submanifold with $T(S\cap L)=TS\cap TL$. Then $S$ can be ``lifted'' to a submanifold of $Bl_L(M)$, namely the closure  of the inverse image of ${S}\setminus L$ under $p$:
$$\overline{\overline{S}}:=\overline{p^{-1}({S}\setminus L)}.$$

Now let $(M_i,Z_i,\Pi_i)$ be {log-symplectic} manifolds, for $i=1,2$.
The product $M_1\times M_2$ is not {log-symplectic} in general\footnote{
However it fits in a slight generalization of the notion of {log-symplectic} structure used in this note: indeed $(M_1\times Z_2) \cup (Z_1\times M_2)$ is a normal crossing divisor, and vector fields tangent to it give rise to a Lie algebroid to which the   Poisson structure on $M_1\times {\hat{M}_2}$ lifts in a non-degenerate way (we thank Aldo Witte for pointing this out to us). One can check that if $f\colon M_1\to M_2$ is a Poisson map transverse to $Z_2$, then $graph(f)$ intersects transversely both $M_1\times Z_2$ and  $Z_1\times M_2$. {This statement generalizes 
Example \ref{ex:sympllog} b) and
can be viewed as an analogue of Proposition \ref{prop:coisoBl}.}}
, but \cite{Polishchuk}, \cite[\S 2.2]{GuLi} 
\begin{equation}\label{blow}
X:=Bl_{Z_1\times Z_2}(M_1\times M_2)\setminus(
\overline{\overline{M_1\times Z_2}} \cup\overline{\overline{Z_1\times M_2}})
\end{equation}
is {log-symplectic} with singular locus the exceptional divisor $p^{-1}(Z_1\times Z_2)$, and the blow-down map $p\colon X\to M_1\times \hat{M_2}$ is Poisson, {where $\hat{M}_2$ denotes
$(M_2,-\Pi_2)$}.

\begin{prop}\label{prop:coisoBl}
Let $f\colon (M_1,Z_1,\Pi_1)\to (M_2,Z_2,\Pi_2)$ be a Poisson map with $f(Z_1)\subset Z_2$. Then  $$\overline{\overline{graph(f)}}$$ 
is a $b$-coisotropic submanifold of {the log-symplectic manifold $X$ defined in \eqref{blow}.}
\end{prop}

\begin{proof}
The intersection $graph(f)\cap Z_1\times Z_2$ is clean, since it coincides with $graph(f|_{Z_1})$ thanks to the assumption $f(Z_1)\subset Z_2$. Hence $graph(f)$ can be ``lifted'' to $X$.

The resulting submanifold $\overline{\overline{graph(f)}}$ is coisotropic: 
$graph(f)$ is coisotropic in $M_1\times {\hat{M}_2}$ because $f$ is a Poisson map, so 
 $p^{-1}(graph(f)\setminus Z_1\times Z_2)$ is coisotropic in $X$ ({since $p$ is a Poisson diffeomorphism away from the exceptional divisor}), and the same holds for its closure. 
 
To finish the proof, we have to show that  $\overline{\overline{graph(f)}}$ is transverse to the exceptional divisor $E:=p^{-1}(Z_1\times Z_2)$. Let $(x^{(1)}_{i})$ be local coordinates on $M_1$ such that  $Z_1=\{x^{(1)}_{1}=0\}$, and similarly let $(x^{(2)}_{j})$ be local coordinates on   $M_2$ such that   $Z_2=\{x^{(2)}_{1}=0\}$. Then $$(x^{(1)}_{i},f^*(x^{(2)}_{j})-x^{(2)}_{j})$$ are local coordinates on $M_1\times M_2$ that are adapted $graph(f)$, but also to
 $$(Z_1\times Z_2)=\{x^{(1)}_{1}=0,f^*(x^{(2)}_{1})-x^{(2)}_{1}=0\},$$ due to the hypothesis $f(Z_1)\subset Z_2$. Hence we can apply Lemma \ref{lem:bltrans}, which yields the desired transversality.
 \end{proof}

The proof of Proposition \ref{prop:coisoBl} uses the following statement, for which 
we could not find a reference in the literature.

\begin{lemma}\label{lem:bltrans}
Let $m,n$ be non-negative integers. Consider $\RR^{n+m}$ with standard coordinates $x_1,\dots,x_n,y_1,\dots,y_{m}$, and the subspaces 
\begin{align*}
Z:=&\{0\}\times \RR^{m},\\
S:=&(\RR^{k}\times \{0\})\times (\RR^{l }\times \{0\})
\end{align*}
  where $k\le n$ and $ l\le m$. Then, in the blow-up $Bl_Z(\RR^{n+m})$, the submanifold $\overline{\overline{S}}$ interesects transversely the exceptional divisor $E$.
\end{lemma}

\begin{proof}
We have $$Bl_Z(\RR^{n+m})=\overline{\Big\{\Big((x,y),[x]\Big): x\in \RR^{n}\setminus\{0\}, y\in \RR^{m}\Big\}}\;\subset \RR^{n+m}\times \RR P^{n-1},$$
where $[\cdot]$ denotes the   class in projective space.
Notice that by taking the closure we are adding exactly the exceptional divisor $$E=\{0\}\times \RR^{m}\times \RR P^{n-1}.$$ We have
$$\overline{\overline{S}}=\overline{\Big\{\Big((x_1,0,y_1,0),[(x_1,0)]\Big): x_1\in \RR^{k}\setminus\{0\}, y_1\in \RR^{l}\Big\}}.
$$ By taking the closure we are adding exactly 
$$\{0\}\times (\RR^{l}\times\{0\})\times \RR P^{k-1}=\overline{\overline{S}}\cap E.$$
For every point $p\in \overline{\overline{S}}\cap E$  there is a curve of the form 
$$\gamma\colon t\mapsto \Big((tx_1,0,y_1,0),[(x_1,0)]\Big)$$
lying in $\overline{\overline{S}}$ with $\gamma(0)=p$, and clearly $\frac{d}{dt}|_0\gamma(t)\notin  T_pE$. 
Since {$\frac{d}{dt}|_0\gamma(t)\in T_p\overline{\overline{S}}$ and} 
$E$ has codimension $1$, we obtain
$T_pE+T_p\overline{\overline{S}}=T_pBl_Z(\RR^{n+m})$.
\end{proof}

\begin{remark}
One can show that for any pair of submanifolds $L$ and $S$ intersecting cleanly, {around any point of the intersection} there exist local coordinates of the ambient manifold $M$ that are simultaneously adapted to both submanifolds. Lemma \ref{lem:bltrans} than implies that, with the notation of the beginning of this subsection,  $\overline{\overline{S}}$ intersects transversely the hypersurface $p^{-1}(L)$ of $Bl_L(M)$.
  \end{remark}

\subsection{$b$-coisotropic embeddings and the $b$-Gotay theorem}
\leavevmode
\vspace{0.15cm}

If $C\overset{i}{\hookrightarrow}(M,Z,\omega)$ is $b$-coisotropic, then Proposition \ref{distribution} implies that $(C,C\cap Z,^bi^*\omega)$ is $b$-presymplectic, i.e. the $b$-two-form $^bi^*\omega\in{}^{b}\Omega^{2}(C)$ is closed of constant rank. Conversely, in this subsection we prove that any $b$-presymplectic manifold embeds $b$-coisotropically into a $b$-symplectic manifold, which is unique up to neighborhood equivalence.
In other words, we show a version of Gotay's theorem for $b$-coisotropic submanifolds. For Lagrangian submanifolds, this becomes a version of Weinstein's tubular neighborhood theorem, which was already obtained in \cite[Theorem 5.18]{CharlotteThesis}.

As a consequence, \emph{a $b$-coisotropic submanifold $C\subset (M,Z,\omega)$ determines $\omega$ (up to $b$-symplectomorphism) in a neighborhood of $C$}.  
Notice that arbitrary coisotropic submanifolds of the log-symplectic manifold $(M,Z,\Pi)$ do not satisfy this property: for instance $Z$ is a coisotropic (even Poisson) submanifold, and by \cite{miranda} the additional data consisting of a certain element of $H^1_{\Pi}(Z)$ is necessary in order to determine the $b$-symplectic structure in a neighborhood of $Z$.

\begin{defi}
A \emph{$b$-presymplectic form} on a $b$-manifold is a $b$-two-form  which is closed and of constant rank.
\end{defi}

\begin{defi}
	Let $(M_{1},Z_{1},\omega)$ be a $b$-manifold endowed with a 
	$b$-presymplectic form
	$\omega\in{}^{b}\Omega^{2}(M_{1})$. A \emph{$b$-coisotropic embedding} of $M_{1}$ into a $b$-symplectic manifold $(M_{2},Z_{2},\Omega)$ is a $b$-map $\phi:(M_{1},Z_{1})\rightarrow(M_{2},Z_{2})$ such that $\phi$ is an embedding and
	\begin{enumerate}[i)]
		\item ${}^{b}\phi^{*}\Omega=\omega$.
		\item $\phi(M_{1})$ is $b$-coisotropic in $(M_{2},Z_{2},\Omega)$.
	\end{enumerate}
\end{defi}

We will prove the following Gotay theorem in the $b$-symplectic setting.

\begin{thm}[The $b$-Gotay theorem]\label{gotay}
	Let $(C,Z_{C},\omega_{C})$ be a $b$-manifold with a 
$b$-presymplectic form
	$\omega_{C}\in{}^{b}\Omega^{2}(C)$. We then have the following:
\begin{itemize}
\item[a)]	 $C$ embeds $b$-coisotropically into a $b$-symplectic manifold,
\item[b)] the embedding is unique up to $b$-symplectomorphism in a tubular neighborhood of $C$, fixing $C$ pointwise.
\end{itemize}
\end{thm}

\noindent
We divide the proof of Theorem \ref{gotay} into several steps. We roughly follow the reasoning from the symplectic case, presented in \cite{Gotay}.
We start by constructing a $b$-symplectic thickening of the $b$-presymplectic manifold $C$, from which item $a)$ of Theorem \ref{gotay} will follow.

\begin{prop}\label{existence}
Denote by $E$ the vector bundle $\text{Ker}(\omega_{C})\subset{}^{b}TC$. Then there is a $b$-symplectic structure $\Omega_G$ on a neighborhood of the zero section $C\subset E^{*}$.
	\end{prop}
	\begin{proof}
		Fix a complement $G$ to $E$ in ${}^{b}TC$, and let  $j:E^{*}\hookrightarrow {}^{b}T^{*}C$ be the induced inclusion.
It is clear that $j(E^{*})=G^{0}$. Since both the bundle projection $\pi:\left(E^{*},E^{*}|_{Z_{C}}\right)\rightarrow (C,Z_{C})$ and the inclusion $j:\left(E^{*},E^{*}|_{Z_{C}}\right)\rightarrow\left({}^{b}T^{*}C,{}^{b}T^{*}C|_{Z_{C}}\right)$ are $b$-maps, we can define a $b$-two-form $\Omega_{G}$ on $(E^{*},E^{*}|_{Z_{C}})$ by
		\begin{equation}\label{Omega}
		\Omega_{G}:={}^{b}\pi^{*}\omega_{C}+{}^{b}j^{*}\omega_{can}.
		\end{equation}
		Here $\omega_{can}$ denotes the canonical $b$-symplectic form on ${}^{b}T^{*}C$ as in  Example \ref{ex:bcotg}, and the subscript $G$ is used to stress that the definition depends on the choice of complement $G$. 
		
		We want to show that $\Omega_{G}$ is $b$-symplectic on a neighborhood of $C\subset \left(E^{*},E^{*}|_{Z_{C}}\right)$. As $\Omega_{G}$ is clearly $b$-closed, it suffices to prove that $\Omega_{G}$ is non-degenerate at points $p\in C$.

		\begin{claim*}
			Under the decomposition
			\begin{equation*}
		{}^{b}T_{p}\left({}^{b}T^{*}C\right)\cong{}^{b}T_{p}C\oplus{}^{b}T_{p}^{*}C,
		\end{equation*}
		of Lemma \ref{decomposition},
			 the canonical $b$-symplectic form is the usual pairing
			\begin{equation}\label{pairing}
			\left(\omega_{can}\right)_{p}\left(v+\alpha,w+\beta\right)=\langle v,\beta\rangle - \langle w,\alpha\rangle.
			\end{equation}
		\end{claim*}
		\noindent
This claim can be checked writing in cotangent coordinates
$
		\omega_{can}=\frac{dx_{1}}{x_{1}}\wedge dy_{1}+\sum_{i=2}^{n}dx_{i}\wedge dy_{i},
$
and noticing that $y_{i}$ is a linear coordinate on each fiber ${}^{b}T_{p}^{*}C$, i.e. $y_{i}\in\left({}^{b}T_{p}^{*}C\right)^{*}\cong{}^{b}T_{p}C$. \hfill $\bigtriangleup$

\vspace{0.3cm}
\noindent
Consider now the decomposition
		\begin{equation}\label{d}
		{}^{b}T_{p}E^{*}\cong{}^{b}T_{p}C\oplus E_{p}^{*}= E_{p}\oplus G_{p}\oplus E_{p}^{*}
		\end{equation}
given by  Lemma \ref{decomposition}.
		Using      Lemma \ref{maps} $b)$  we have $\left({}^{b}j_{*}\right)_{p}=\text{Id}_{\hspace{0.05cm} {}^{b}T_{p}C}\oplus\left.j\right|_{E_{p}^{*}}$.
		Hence under the  decomposition \eqref{d}
		 we have
		\begin{align*}
		\left({}^{b}j^{*}\omega_{can}\right)_{p}(v+w+\alpha,x+y+\beta)&=\left(\omega_{can}\right)_{p}\big(v+w+j(\alpha),x+y+j(\beta)\big)\\
		&=\langle v+w,j(\beta)\rangle - \langle x+y,j(\alpha)\rangle\\
		&=\langle v,j(\beta)\rangle-\langle x,j(\alpha)\rangle,
		\end{align*}
		using the above claim and recalling that $j(E_{p}^{*})=G_{p}^{0}$.
		In matrix notation,
		\begin{equation}\label{block1}
		\left({}^{b}j^{*}\omega_{can}\right)_{p}=\begin{blockarray}{cccc}
		& E_{p} & G_{p} & E_{p}^{*} \\
		\begin{block}{c(ccc)}
		E_{p} & 0 & 0 & A \\
		G_{p} & 0 & 0 & 0 \\
		E_{p}^{*} & -A & 0 & 0 \\
		\end{block}
		\end{blockarray},
		\end{equation}
		for some matrix $A$ of full rank.
		Similarly we have $\left({}^{b}\pi_{*}\right)_{p}=\text{Id}_{\hspace{0.05cm} {}^{b}T_{p}C}\oplus0$,
 applying Lemma \ref{maps} $b)$ to $\pi$ (regarded as a vector bundle map).	
		Therefore, under \eqref{d} we get
		\[
		\left({}^{b}\pi^{*}\omega_{C}\right)_{p}(v+w+\alpha,x+y+\beta)=\left(\omega_{C}\right)_{p}(v+w,x+y),
		\]
		so that we get a matrix representation of the form
		\begin{equation}\label{block2}
		\left({}^{b}\pi^{*}\omega_{C}\right)_{p}=\begin{blockarray}{cccc}
		& E_{p} & G_{p} & E_{p}^{*} \\
		\begin{block}{c(ccc)}
		E_{p} & 0&0 & 0 \\
		G_{p} & 0& B & 0 \\
		E_{p}^{*} & 0 & 0 & 0 \\
		\end{block}
		\end{blockarray},
		\end{equation}
		where we also use that $E=\text{Ker}(\omega_{C})$. Note that the matrix $B$ in \eqref{block2} is of full rank since the restriction of $\left(\omega_{C}\right)_{p}$ to $G_{p}$ is non-degenerate. Combining \eqref{block1} and \eqref{block2}, we have that
		\begin{equation}\label{matrix}
		\left(\Omega_{G}\right)_{p}=\left({}^{b}\pi^{*}\omega_{C}\right)_{p}+\left({}^{b}j^{*}\omega_{can}\right)_{p}=\begin{blockarray}{cccc}
		& E_{p} & G_{p} & E_{p}^{*} \\
		\begin{block}{c(ccc)}
		E_{p} & 0&0 & A \\
		G_{p} & 0& B & 0 \\
		E_{p}^{*} & -A & 0 & 0 \\
		\end{block}
		\end{blockarray}\;,
		\end{equation}
		which is of maximal rank. Therefore, $\Omega_{G}$ is non-degenerate at points $p\in C\subset E^{*}$.
	\end{proof}

\begin{proof}[Proof of item a) of Theorem \ref{gotay}]
We show that the inclusion $(C,Z_{C},\omega_{C})\overset{i}\hookrightarrow(E^{*},E^{*}|_{Z_{C}},\Omega_{G})$ is indeed a $b$-coisotropic embedding, i.e.
\begin{enumerate}[i)]
	\item ${}^{b}i^{*}\Omega_{G}=\omega_{C}$,
	\item ${}^{b}TC^{\Omega_{G}}\subset{}^{b}TC$.
\end{enumerate}
We have ${}^{b}i^{*}\Omega_{G}={}^{b}(\pi\circ i)^{*}\omega_{C}+{}^{b}(j\circ i)^{*}\omega_{can}=\omega_{C}+{}^{b}(j\circ i)^{*}\omega_{can}$. Note that $j\circ i$ is the inclusion of $C$ into ${}^{b}T^{*}C$, so that ${}^{b}(j\circ i)^{*}\omega_{can}=0$ since $C$ is $b$-Lagrangian in $\left({}^{b}T^{*}C,\omega_{can}\right)$. 

\noindent
To check ii), we let $p\in C$ and choose $v+w+\alpha\in E_{p}\oplus G_{p}\oplus E_{p}^{*}\cong{}^{b}T_{p}E^{*}$ lying in ${}^{b}T_{p}C^{\Omega_{G}}$. Let $x\in E_{p}\subset{}^{b}T_{p}C$ be arbitrary. Thanks to \eqref{matrix}, we then have
\[
0=\left(\Omega_{G}\right)_{p}(x,v+w+\alpha)=\left(\Omega_{G}\right)_{p}(x,\alpha),
\]
which forces that $\alpha=0$ due to non-degeneracy of $\left(\Omega_{G}\right)_{p}$ on $E_{p}\times E_{p}^{*}$. Hence $v+w+\alpha=v+w$ lies in $E_{p}\oplus G_{p}={}^{b}T_{p}C$, as desired. 
\end{proof}

The uniqueness statement $b)$ of Theorem \ref{gotay} is an immediate consequence of the following proposition, to which we devote the rest of this subsection.

\begin{prop}\label{normal} Let $(M,Z,\omega)$ be a $b$-symplectic manifold and
	$C$
	a $b$-coisotropic submanifold, with induced $b$-presymplectic form $\omega_{C}\in{}^{b}\Omega^{2}(C)$. Let $E:=\text{Ker}(\omega_{C})$ and fix a splitting ${}^{b}TC=E\oplus G$. Then there is a $b$-symplectomorphism $\tau$ between a neighborhood of $C\subset (M,Z,\omega)$ and a neighborhood of $C\subset (E^{*},E^{*}|_{C\cap Z},\Omega_{G})$, with $\tau|_{C}=\text{Id}_{C}$.
	\end{prop}
	\begin{proof}
		Since $\omega|_{G\times G}$ is non-degenerate, we have a decomposition ${}^{b}TM|_{C}=G\oplus G^{\omega}$ as symplectic vector bundles. Note that $E$ is a Lagrangian subbundle of $(G^{\omega},\omega)$, since
		\begin{equation}\label{orthogonal}
		E^{\omega}\cap G^{\omega}=\left(E\oplus G\right)^{\omega}={}^{b}TC^{\omega}={}^{b}TC^{\omega}\cap{}^{b}TC=E.
		\end{equation}
We fix  a Lagrangian complement $V$ to $E$ in $(G^{\omega},\omega)$, i.e. $G^{\omega}=E\oplus V$. 

The idea of the proof is to construct a $b$-diffeomorphism between neighborhoods of $C$ in $M$ and $E^*$ -- obtained as a composition of $b$-diffeomorphisms to a neighborhood in $V$ --   
whose $b$-derivative at points of $C$ pulls back $\Omega_{G}$ to $\omega$, and then apply a Moser argument.

We start by establishing a $b$-geometry version of the tubular neighborhood theorem, in which $V$ plays the role of the normal bundle to $C$.
		
\begin{claim}\label{bdiffeo}
			There is a $b$-diffeomorphism $\phi$ between a neighborhood of $C$ in $(V,V|_{C\cap Z})$ and
			a neighborhood of $C$ in $(M,Z)$, satisfying $\left.{}^{b}\phi_{*}\right|_{C}=\text{Id}_{\hspace{0.05cm}{}^{b}TM|_{C}}$.
		\end{claim}
		\noindent
		We will construct this map in two steps:
		\begin{equation*}
		V\xrightarrow[(1)]{} \rho(V)\xrightarrow[(2)]{} M
		\end{equation*}
		
		\noindent
		\textit{Step 1.} Let $\rho:{}^{b}TM\rightarrow TM$ denote the anchor map of ${}^{b}TM$ and notice that its restriction to $V$ is injective. To see this, recall the decomposition 
		\begin{equation}\label{count}
		{}^{b}TM|_{C}=G\oplus G^{\omega}=G\oplus E\oplus V={}^{b}TC\oplus V
		\end{equation}
		and the fact that $\text{Ker}(\rho|_{C})\subset{}^{b}TC$ by Lemma \ref{L}, so that $\text{Ker}(\rho)$ intersects $V$ trivially. As such, we get a $b$-diffeomorphism $\rho:\left(V,V|_{C\cap Z}\right)\rightarrow\left(\rho(V),\rho(V)|_{C\cap Z}\right)$.
		
		\vspace{0.3cm}
		\noindent
		\textit{Step 2.}
		The distribution $\rho(V)$ is complementary to $TC$, i.e. 
		\[
		TM|_{C}=TC\oplus\rho(V).
		\]
		Indeed, by Step 1, we have at any point $p\in C$
		\begin{equation*}
		\dim(T_{p}M)=\dim(T_{p}C)+\dim(V_{p})
		=\dim(T_{p}C)+\dim(\rho(V_{p})),
		\end{equation*}
		and moreover, if $v\in V_{p}$ is such that $\rho(v)\in T_{p}C$, then $v\in {}^{b}T_{p}C\cap V_{p}=\{0\}$. Now fix a Riemannian metric
		 $g$ on $M$ such that $Z\subset (M,g)$ is totally geodesic (e.g. \cite[Lemma 6.8]{Milnor}).
		The corresponding exponential map $\exp_{g}$ takes a neighborhood of $C\subset\rho(V)$ diffeomorphically onto a neighborhood of $C\subset M$. Moreover the fibers of $\rho(V)$ over $C\cap Z$ are mapped into $Z$, since $\rho(V_{p})\in T_{p}Z$ for $p\in C\cap Z$ and $Z$ is totally geodesic. Therefore, the map\footnote{Alternatively, one can apply \cite[Example 3.3.9, p. 88-89]{CandelConlonI} (see also \cite[Theorem 2]{SevestreWurzbacher}).} $\exp_{g}:\left(\rho(V),\rho(V)|_{C\cap Z}\right)\rightarrow (M,Z)$ is a  $b$-diffeomorphism between neighborhoods of $C$.

We now show that $\phi:=\exp_{g}\circ\rho\colon V\to M$ has the claimed property. That is, we show that $\left.\left[{}^{b}(\exp_{g}\circ\rho)_{*}\right]\right|_{C}$ is the identity map on ${}^{b}TV|_{C}\cong{}^{b}TC\oplus V={}^{b}TM|_{C}$,
		by checking that it acts as the identity on sections. We will need the commutative diagram
		\begin{equation}\label{diagramma}
		\begin{tikzcd}[column sep=large, row sep=large]
		{}^{b}TC\oplus V\arrow{rr}{{}^{b}(\exp_{g}\circ\rho)_{*}}\arrow{d}{\rho\oplus\text{Id}_{V}} & & {}^{b}TM|_{C}\arrow{d}{\rho}\\
		TC\oplus V\arrow{rr}{(\exp_{g}\circ\rho)_{*}}& & TM|_{C}
		\end{tikzcd},
		\end{equation}
		which implicitly uses $a)$ of Lemma \ref{maps}. We will also use that for all $q\in C$ the ordinary derivative reads 
		\begin{equation}\label{derivative}
		\left[(\exp_{g}\circ\rho)_{*}\right]_{q}:T_{q}V\cong T_{q}C\oplus V_{q}\rightarrow T_{q}M=T_{q}C\oplus\rho(V_{q}):w+v\mapsto w+\rho(v).
		\end{equation}
		For a section $X+Y\in\Gamma\left({}^{b}TC\oplus V\right)$ we now compute
		\begin{align*}
		\rho\left[{}^{b}(\exp_{g}\circ\rho)_{*}(X+Y)\right]&=(\exp_{g}\circ\rho)_{*}(\rho(X)+Y)\\
		&=\rho(X)+\rho(Y)\\
		&=\rho(X+Y)
		\end{align*}
		using \eqref{diagramma} in the first equality and \eqref{derivative} in the second.
		Since the anchor $\rho$ is injective on sections, this implies that ${}^{b}(\exp_{g}\circ\rho)_{*}(X+Y)=X+Y$, as desired. Claim \ref{bdiffeo} is proved.	
		\hfill $\bigtriangleup$
		
		\vspace{0.5cm}
		\noindent
		Next, the map $$\psi:V\rightarrow E^{*},\;v\mapsto \left.-\iota_{v}\omega\right|_{E}$$ is an isomorphism of vector bundles covering $\text{Id}_{C}$, whence a $b$-diffeomorphism between the total spaces (For the injectivity, note   that $\left.\iota_{v}\omega\right|_{E}=0$ implies that $v\in E^{\omega}\cap G^{\omega}=E$ as in \eqref{orthogonal}, so that $v\in V\cap E=\{0\}$). The composition $\psi\circ\phi^{-1}:(M,Z)\rightarrow (E^{*},E^{*}|_{C\cap Z})$ is a $b$-diffeomorphism between neighborhoods of $C$, with $\left.\left(\psi\circ\phi^{-1}\right)\right|_{C}=\text{Id}_{C}$.
		
		\begin{claim}
This $b$-diffeomorphism  satisfies		$\left.\left[{}^{b}(\psi\circ\phi^{-1})^{*}\Omega_{G}\right]\right|_{C}=\left.\omega\right|_{C}$.
		\end{claim}
		\noindent
		As before, let $\pi:E^{*}\rightarrow C$ denote the bundle projection, and let $j:E^{*}\hookrightarrow{}^{b}T^{*}C$ be the inclusion induced by the splitting ${}^{b}TC=E\oplus G$. 
				Since $\psi\colon V\rightarrow E^{*}$ is a vector bundle morphism covering $\text{Id}_C$,
by Lemma \ref{maps} b) we have  that $$ {}^{b}\psi_{*}|_C\colon
{}^{b}TV|_C\cong{}^{b}TC\oplus V\to {}^{b}TE^*|_C\cong{}^{b}TC\oplus E^* $$
equals $ Id_{\hspace{0.05cm}^{b}TC} \oplus \psi$.
Furthermore $\left.{}^{b}\phi_{*}\right|_{C}=\text{Id}_{\hspace{0.05cm}{}^{b}TM|_{C}}$
by Claim 1.  	
Therefore, for $p\in C$ and $x_i+v_i\in  {}^{b}T_pC \oplus V_{p}={}^{b}T_{p}M$, we have
\begin{equation}\label{eq}
\left[{}^{b}(\psi\circ\phi^{-1})^{*}\Omega_{G}\right]_{p}(x_{1}+v_{1},x_{2}+v_{2})=\left(\Omega_{G}\right)_{p}\big(x_{1}+\psi(v_{1}),x_{2}+\psi(v_{2})\big).
\end{equation}
Recalling equation \eqref{Omega} and applying Lemma \ref{maps} $b)$ as in the proof of Proposition \ref{existence}, we expand the right hand side of \eqref{eq} as follows:
		\begin{align*}
		\left(\Omega_{G}\right)_{p}\big(x_{1}+\psi(v_{1}),x_{2}+\psi(v_{2})\big)&=\omega_{p}(x_{1},x_{2})+\left(\omega_{can}\right)_{p}(x_{1}+j(\psi(v_{1})),x_{2}+j(\psi(v_{2})))\\
		&=\omega_{p}(x_{1},x_{2})+\left\langle x_{1},j(\psi(v_{2}))\right\rangle-\left\langle x_{2},j(\psi(v_{1}))\right\rangle\\
		&=\omega_{p}(x_{1},x_{2})+\langle e_{1},\psi(v_{2})\rangle - \langle e_{2},\psi(v_{1})\rangle\\
		&=\omega_{p}(x_{1},x_{2})+\omega_{p}(e_{1},v_{2})+\omega_{p}(v_{1},e_{2})\\
		&=\omega_{p}(x_{1}+v_{1},x_{2}+v_{2}),
		\end{align*}
		using equation \eqref{pairing} in the second equality, 
writing $x_i=e_i+g_i\in E_{p}\oplus G_{p}={}^{b}T_{p}C$, 		
		and			using	in the last equality that $V$ is a Lagrangian subbundle of $(G^{\omega},\omega)$. This finishes the proof of Claim 2.	\hfill $\bigtriangleup$
		
		\vspace{0.4cm}
		\noindent
Applying Proposition \ref{moser} (relative $b$-Moser) yields a $b$-diffeomorphism $\rho$, defined on a neighborhood of $C\subset (M,Z)$, such that ${}^{b}\rho^{*}\left({}^{b}(\psi\circ\phi^{-1})^{*}\Omega_{G}\right)=\omega$ and $\rho|_{C}=\text{Id}_{C}$. So setting $\tau:=\psi\circ\phi^{-1}\circ\rho$ finishes the proof.
	\end{proof}

\section{Strong $b$-coisotropic submanifolds and  $b$-symplectic reduction}\label{sec:strongbcoiso}

We consider a subclass of $b$-coisotropic submanifolds in  $b$-symplectic manifolds, namely, the coisotropic submanifolds that are transverse to the symplectic leaves they meet. The main observation is that their characteristic distribution has constant rank, and the quotient  (whenever smooth) by this distribution inherits a $b$-symplectic form (Proposition \ref{prop:bcoisored}).

\subsection{Strong $b$-coisotropic submanifolds}
\leavevmode
\vspace{0.15cm}

In Subsection \ref{bcoiso} we have seen that a $b$-coisotropic submanifold $C\subset(M,Z,\omega)$ comes with a  characteristic distribution 
$$D:=\rho\left({}^{b}TC^{\omega}\right)=\Pi^{\sharp}\left(TC^{0}\right).$$ In general, $D$ fails to be regular. To force that $D$ has constant rank, we have to impose a condition on $C$ that is stronger than $b$-coisotropicity.

\begin{defi}\label{def:strongbcoiso}
A submanifold $C$ of a log-symplectic manifold $(M,Z,\Pi)$ is called \emph{strong $b$-coisotropic} if it is  coisotropic (with respect to $\Pi$) and transverse to all the symplectic leaves of $ (M,\Pi)$ it meets.
\end{defi}

To justify this definition, we note that
\begin{align}\label{annihilator}
\left.\Pi_{p}^{\sharp}\right|_{T_{p}C^{0}}\ \text{is injective}&\Leftrightarrow \text{Ker}\left(\Pi_{p}^{\sharp}\right)\cap T_{p}C^{0}=\{0\}\nonumber\\
&\Leftrightarrow T_{p}\mathcal{O}^{0}\cap T_{p}C^{0}=\{0\}\nonumber\\
&\nonumber\Leftrightarrow \left(T_{p}\mathcal{O}+T_{p}C\right)^{0}=\{0\}\\ \
&\Leftrightarrow T_{p}\mathcal{O}+T_{p}C=T_{p}M,
\end{align}
where $\mathcal{O}$ denotes the symplectic leaf through $p$. The last equation is exactly the transversality condition of Definition \ref{def:strongbcoiso}. Consequently, we have:

\begin{prop}\label{prop:strongrank}
{Let $C\subset (M,Z,\Pi)$ be a coisotropic submanifold. Then $C$ is
strong $b$-coisotropic if{f} the  characteristic distribution of $C$ is regular, with rank equal to $codim(C)$.}
\end{prop}

Lemma \ref{lem:twochar} immediately implies:

\begin{cor}\label{cor:D}
Let $C\subset (M,Z,\omega)$ be
strong $b$-coisotropic. Then its characteristic distribution is tangent to $Z$,
and corresponds to ${}^{b}TC^{\omega} $ under the bijection of Lemma \ref{lem:distrsplit} b).
\end{cor}

\begin{remark}
If $C$ is a strong $b$-coisotropic submanifold of $(M^{2n},Z,\Pi)$ {intersecting $Z$}, then necessarily $dim(C)\geq n+1$. 
Indeed, if $\mathcal{O}$ denotes the symplectic leaf through $p\in C\cap Z$, then we have
\begin{align*}
\dim(C)&=\dim\left(T_{p}\mathcal{O}+T_{p}C\right)+\dim\left(T_{p}\mathcal{O}\cap T_{p}C\right)-\dim(\mathcal{O})\\
&=\dim\left(T_{p}\mathcal{O}\cap T_{p}C\right)+2\\
&\geq n+1,
\end{align*}
where the last inequality holds since $T_{p}\mathcal{O}\cap T_{p}C$ is a coisotropic subspace of  {the $(2n-2)$-dimensional vector space} $T_{p}\mathcal{O}$.
Alternatively, one can observe that a middle-dimensional $b$-coisotropic submanifold $C^{n}\subset (M^{2n},Z,\omega)$ is $b$-Lagrangian (i.e. ${}^{b}TC^{\omega}={}^{b}TC$). Its characteristic distribution satisfies
\[
\dim\left(D_{p}\right)=\begin{cases}
\dim(C)-1\hspace{1.2cm} \text{if}\ p\in C\cap Z\\
\dim(C) \hspace{1.83cm} \text{else}
\end{cases},
\]
so that $C$ cannot be strong $b$-coisotropic
{whenever it intersects $Z$, due to Proposition \ref{prop:strongrank}.}
\end{remark}

\subsection{Coisotropic reduction in $b$-symplectic geometry}
\leavevmode
\vspace{0.15cm}

In this subsection we  adapt coisotropic 
reduction to the $b$-symplectic category.
It is well-known that, given a coisotropic submanifold $C$ of a Poisson manifold $M$, its quotient $\underline{C}$ by the characteristic distribution is again a Poisson manifold, provided it is smooth. More precisely, the vanishing ideal
$\mathcal{I}_{C}$ is a Poisson subalgebra of $\left(C^{\infty}(M),\{\cdot,\cdot\}\right)$, and denoting by $\mathcal{N}(\mathcal{I}_{C}):=\{f\in C^{\infty}(M): \{f,\mathcal{I}_{C}\}\subset \mathcal{I}_{C}\}$ its Poisson normalizer, we have that $\mathcal{N}(\mathcal{I}_{C})/\mathcal{I}_{C}$ is a Poisson algebra. As an algebra it is canonically isomorphic to the algebra of smooth functions on the quotient $\underline{C}$, so it endows the latter with a Poisson structure, called the \emph{reduced Poisson structure}.

\begin{remark}\label{rem:samered}
When the Poisson structure on $M$ is non-degenerate, i.e. corresponds to a symplectic form $\omega\in \Omega^2(M)$, the reduced Poisson structure on $\underline{C}$ is also non-degenerate. Indeed \cite{Redsingmomap}, it corresponds to the symplectic form $\omega_{red}$ on $\underline{C}$ obtained by symplectic coisotropic reduction, i.e. the unique one that satisfies $q^*\omega_{red}=i^*\omega$, where $q\colon C\to \underline{C}$ is the projection and $i\colon C\to M$ is the inclusion. 
\end{remark}

 \begin{prop}[Coisotropic reduction]\label{prop:bcoisored}
Let $C$ be a strong $b$-coisotropic submanifold of a $b$-symplectic manifold $(M,Z,\omega,\Pi)$. Then
$D:=\Pi^{\sharp}\left(TC^{0}\right)$ is a (constant rank) involutive distribution on $C$. 
Assume that $\underline{C}:=C/D$ has a smooth manifold structure, such that the projection $q:C\rightarrow\underline{C}$ is a submersion.  
Then $\underline{C}$ inherits a $b$-symplectic structure $\underline{\Omega}$, determined by 
\begin{equation}\label{redform} 
{}^{b}q^{*}\underline{\Omega}={}^{b}i^{*}\omega,
\end{equation} 
where $i:C\hookrightarrow M$ is the inclusion. Its corresponding log-symplectic structure is exactly
the reduced Poisson structure on  
$\underline{C}$ obtained
from $\Pi$.
\end{prop}
\begin{proof}
We know that $D$ has constant rank, by Proposition \ref{prop:strongrank}. As for involutivity, first note that $D$ is generated by Hamiltonians $\left.X_{h}\right|_{C}$ of functions $h\in\mathcal{I}_{C}$. On such generators, we have
\[
\left[\left.X_{h_{1}}\right|_{C},\left.X_{h_{2}}\right|_{C}\right]=\left.\left[X_{h_{1}},X_{h_{2}}\right]\right|_{C}=\left.X_{\{h_{1},h_{2}\}}\right|_{C},
\]
where $\{h_{1},h_{2}\}\in\mathcal{I}_{C}$ due to coisotropicity of $C$. Hence $D$ is involutive.

The quotient $\underline{C\cap Z}:=\left(C\cap Z\right)/D$ {is a smooth submanifold of $\underline{C}$, since for every slice $S$ in $C$ transverse to $D$, the intersection $S\cap Z$ is  a smooth slice in $C\cap Z$ transverse to $D$.
The leaf space   $(\underline{C},\underline{C\cap Z})$ is  a $b$-manifold}, and the projection $q:(C,C\cap Z)\rightarrow(\underline{C},\underline{C\cap Z})$ is a $b$-map. For $p\in C$, we have an exact sequence
\begin{equation*} 
0\rightarrow D_{p}\hookrightarrow T_{p}C\overset{(q_{*})_{p}}{\longrightarrow} T_{q(p)}\underline{C}\rightarrow 0,
\end{equation*}
which corresponds with an exact sequence on the level of $b$-tangent spaces
\begin{equation}\label{exactsequence} 0\rightarrow\left({}^{b}T_{p}C\right)^{\omega_{p}}\hookrightarrow{}^{b}T_{p}C\overset{\left({}^{b}q_{*}\right)_{p}}{\longrightarrow}{}^{b}T_{q(p)}\underline{C}\rightarrow 0.
\end{equation}
To see this, consider the canonical splitting $\sigma:D\rightarrow{}^{b}TC$ of the anchor $\rho:{}^{b}TC\rightarrow TC$, as constructed in Lemma \ref{lem:distrsplit} $a)$, 
and notice that 
\[
\text{Ker}\left(\left({}^{b}q_{*}\right)_{p}\right)=\sigma\left(\text{Ker}\left(q_{*}\right)_{p}\right)=\sigma(D_{p})=\left({}^{b}T_{p}C\right)^{\omega_{p}},
\]
where the first equality holds by Corollary \ref{split} and the third by Corollary \ref{cor:D}.
 
Since $q$ is a surjective submersion, it admits sections, hence for every sufficiently small open subset $U\subset \underline{C}$ there is a submanifold $S\subset C$ transverse to $D$ such that $q|_{S}\colon S\to U$ is a diffeomorphism.
At points $p\in S$ we have	
$$	{}^{b}T_{p}C
	=\left({}^{b}T_{p}C\right)^{\omega_{p}}\oplus{}^{b}T_{p}S
	$$
due to the sequence \eqref{exactsequence}. This implies that ${}^{b}i_{S}^{*}\omega_{C}$ is a $b$-symplectic form on $S$, where $i_{S}:S\hookrightarrow C$ is the inclusion and $\omega_{C}$  is the restriction of $\omega$ to $C$.
 Denote by $\tau:U\rightarrow S$ the inverse of $q|_{S}:S\rightarrow U$.
Then $\underline{\Omega}:={}^{b}\tau^{*}\left({}^{b}i_{S}^{*}\omega_{C}\right)$ is $b$-symplectic on $U$.
 Away from $\underline{C\cap Z}$,
	this $b$-2-form  agrees with the symplectic form obtained  by symplectic coisotropic reduction from $\omega|_{M\setminus Z}$.
	Denote by $-\underline{\Omega}^{-1}$ the non-degenerate $b$-bivector on $U$ corresponding  to $\underline{\Omega}$. 
	Away from $\underline{C\cap Z}$, the log-symplectic structure $\underline{\rho}(-\underline{\Omega}^{-1})$
	agrees with the reduced Poisson structure, by Remark \ref{rem:samered}. By continuity, the same is true on the whole of $U$. As $U$ was arbitrary, the reduced
	Poisson structure on $\underline{C}$ is log-symplectic, and the above reasoning shows that the corresponding $b$-symplectic form   satisfies equation \eqref{redform}.
\end{proof}

\begin{exs}
\begin{enumerate}[a)]
\item Let $i\colon B\hookrightarrow (M,Z)$ be a $b$-submanifold.
A quick check in coordinates shows\footnote{The converse is also true. If $^bT^*M|_B$ is strong $b$-coisotropic in $^bT^*M$, then $^bT^*M|_B$ is transverse to $^bT^*M|_Z$, which implies that $B$ is transverse to $Z$, i.e. that $B$ is a $b$-submanifold.} that $^bT^*M|_B$ is strong $b$-coisotropic in $^bT^*M$. Its quotient $\underline{^bT^*M|_B}$ is canonically $b$-symplectomorphic to ${}^{b}T^{*}B$. To see this, consider the surjective submersion
\[
\varphi:{}^{b}T^*M|_B\rightarrow{}^{b}T^{*}B:\alpha_{p}\mapsto \left({}^{b}i_{*}\right)_{p}^{*}\alpha_{p}
\]
and notice that the fibers of $\varphi$ coincide with the leaves of the characteristic distribution on $^bT^*M|_B$. So we get a $b$-diffeomorphism $\overline{\varphi}:\underline{^bT^*M|_B}\rightarrow{}^{b}T^{*}B$. To see that this is in fact a $b$-symplectomorphism, we note that the tautological $b$-one-forms on ${}^{b}T^{*}M$ and ${}^{b}T^{*}B$ are related by 
\begin{equation}\label{taut}
{}^{b}\varphi^{*}\theta_{B}={}^{b}j^{*}\theta_{M},
\end{equation}
where $j:^bT^*M|_B\hookrightarrow{}^{b}T^{*}M$ is the inclusion. Recall that the $b$-symplectic form $\underline{\Omega}$ on $\underline{^bT^*M|_B}$ is determined by the relation ${}^{b}q^{*}\underline{\Omega}={}^{b}j^{*}\omega_{M}$, where $q:^bT^*M|_B\rightarrow \underline{^bT^*M|_B}$ is the projection (cf. \eqref{redform}). Hence to conclude that $\overline{\varphi}$ is $b$-symplectic, we have to show that ${}^{b}q^{*}\left({}^{b}\overline{\varphi}^{*}\omega_{B}\right)={}^{b}j^{*}\omega_{M}$. But this is immediate from \eqref{taut} since $\overline{\varphi}\circ q=\varphi$.

\item Given a $b$-manifold $(M,Z)$, let $K$ be a distribution on $M$ tangent to $Z$.  
{Thanks to Lemma \ref{lem:distrsplit} a) we can}
view $K$ as a subbundle $\sigma(K)$ of $^bTM$. Its annihilator $\sigma(K)^{0}$ is strong $b$- coisotropic in $^bT^*M$, and the quotient $\underline{\sigma(K)^{0}}$ is $^bT^*(M/K)$, whenever $M/K$ is smooth. We give a proof of this fact in the particular case of a Hamiltonian group action, see Corollary \ref{ex:bT*group}.
\end{enumerate}
\end{exs}  

\subsection{Moment map reduction in $b$-symplectic geometry}
\leavevmode
\vspace{0.15cm}

Recall that, given an action of a Lie group $G$ on a Poisson manifold $(M,\Pi)$, a  \emph{moment map} is a Poisson map $J:M\rightarrow\mathfrak{g}^{*}$ satisfying
\begin{equation}\label{moment}
\Pi^{\sharp}\left(dJ^{x}\right)=v_{x}\hspace{0.5cm}\forall x\in\mathfrak{g}.
\end{equation}
Here $J^{x}\colon M\rightarrow\mathbb{R}: p\mapsto\langle J(p),x\rangle$ is the $x$-component of $J$, the vector field $v_{x}$ is the infinitesimal generator of the action corresponding with $x\in\mathfrak{g}$, i.e.
\[
{v_{x}(p)=\left.\frac{d}{dt}\right|_{t=0}\exp(-tx)\cdot p,}
\]
 and $\mathfrak{g}^{*}$ is endowed with its canonical Lie-Poisson structure \cite[Section 3]{models}.
A $G$-equivariant map $J:M\rightarrow\mathfrak{g}^{*}$ satisfying \eqref{moment} is automatically Poisson {\cite[Proposition 7.30]{vaisman}}.

In view of Proposition \ref{prop:bcoisored}, we recall a general fact about equivariant moment maps.
 
\begin{lemma}\label{poismoment} 
Let $G$ be a Lie group acting on a Poisson manifold $(M,\Pi)$ with  equivariant moment map $J:M\rightarrow\mathfrak{g}^{*}$. Assume the action is free on $J^{-1}(0)$. Then
\begin{enumerate}[a)]
	\item $J^{-1}(0)$ is a coisotropic submanifold of $(M,\Pi)$.
	\item $J^{-1}(0)$ is transverse to all symplectic leaves of $(M,\Pi)$ it meets. 
	\item the characteristic distribution $\Pi^{\sharp}\left(T(J^{-1}(0))^{0}\right)$ on $J^{-1}(0)$ coincides with the tangent distribution to the orbits of $G\curvearrowright J^{-1}(0)$.
\end{enumerate}
\end{lemma}

\begin{remark}  (i) When $(M,\Pi)$ is a  log-symplectic manifold, Lemma \ref{poismoment} implies that the level set  $J^{-1}(0)$
is a strong $b$-coisotropic submanifold.

(ii) When $G$ a torus, there is a more flexible notion of moment map \cite[Definition 22]{ToricB} for log-symplectic manifolds. The smooth level sets of such moment maps are not strong $b$-coisotropic submanifolds in general. Indeed they can even fail to be transverse to the degeneracy locus $Z$ (see  \cite[Example 23]{ToricB} for an instance where $Z$ itself is such a level set).  
 
\end{remark}

For the sake of  for completeness we provide a proof of Lemma \ref{poismoment}. Items a) and c) also follow from well-known facts in symplectic geometry, by restricting the $G$-action to each symplectic leaf {(whenever $G$ is connected) and using item b).}

\begin{proof}
\begin{enumerate}[a)]
\item We show that $0$ is a regular value of $J$. Choosing $p\in J^{-1}(0)$, it is enough to prove that the restriction $d_{p}J:\text{Im}(\Pi_{p}^{\sharp})\subset T_{p}M\rightarrow\mathfrak{g}^{*}$ is surjective. To this end, assume that $\xi\in\mathfrak{g}$ annihilates $d_{p}J(\text{Im}\big(\Pi_{p}^{\sharp})\big)$. We then get for all $\alpha\in T_{p}^{*}M$ that
\[
\left\langle\alpha,(v_{\xi})_{p}\right\rangle=\left\langle\alpha,\Pi_{p}^{\sharp}\left(d_{p}J^{\xi}\right)\right\rangle=-\left\langle d_{p}J^{\xi},\Pi_{p}^{\sharp}(\alpha)\right\rangle=-\left\langle d_{p}J\left(\Pi_{p}^{\sharp}(\alpha)\right),\xi\right\rangle=0,
\]
and therefore $\left(v_{\xi}\right)_{p}=0$. Since the action $G\curvearrowright J^{-1}(0)$ is free, this implies that $\xi=0$. It follows that $d_{p}J\big(\text{Im}(\Pi_{p}^{\sharp})\big)=\mathfrak{g}^{*}$, so $0$ is indeed a regular value of $J$. In particular, $J^{-1}(0)$ is a submanifold of $M$. The coisotropicity of $J^{-1}(0)$ follows since it is the preimage of a symplectic leaf $\{0\}\subset\mathfrak{g}^{*}$ under a Poisson map.
\item Let $\mathcal{O}$ denote the symplectic leaf through $p\in J^{-1}(0)$. By the computation \eqref{annihilator}, it suffices to prove that $\left.\Pi_{p}^{\sharp}\right|_{\left[T_{p}J^{-1}(0)\right]^{0}}$ is injective. Since $0$ is a regular value, this annihilator is given by
$\left[T_{p}J^{-1}(0)\right]^{0}\nonumber
=\left\{ d_{p}J^{x}: x\in\mathfrak{g}\right\}.$
We now have a composition of maps
\begin{align*}
&\mathfrak{g}\longrightarrow \left[T_{p}J^{-1}(0)\right]^{0} \longrightarrow\Pi_{p}^{\sharp}\left(\left[T_{p}J^{-1}(0)\right]^{0}\right)\\
&x\hspace{0.15cm}\mapsto\hspace{0.3cm} d_{p}J^{x}\hspace{0.2cm}\mapsto\hspace{0.2cm}\Pi_{p}^{\sharp}\left(d_{p}J^{x}\right)=\left(v_{x}\right)_{p},
\end{align*}
that is injective by freeness of $G\curvearrowright J^{-1}(0)$. In particular, $\left.\Pi_{p}^{\sharp}\right|_{\left[T_{p}J^{-1}(0)\right]^{0}}$ is injective.
\item 
We have
\[
\Pi_{p}^{\sharp}\left(\left[T_{p}J^{-1}(0)\right]^{0}\right)=\left\{ \Pi_{p}^{\sharp}\left(d_{p}J^{x}\right): x\in\mathfrak{g}\right\}=\left\{ (v_{x})_{p}: x\in\mathfrak{g}\right\},
\]
which is exactly the tangent space of the $G$-orbit through $p$.
\end{enumerate}
\end{proof}

Combining Proposition \ref{prop:bcoisored} with Lemma \ref{poismoment}, 
we obtain a moment map reduction statement in the $b$-symplectic category. The 
  case $G=S^{1}$ was already addressed in \cite[Proposition 7.8]{gualtieri}.
\begin{cor}[Moment map reduction]\label{reduction}
Consider an action of a {connected} Lie group $G$ on a $b$-symplectic manifold $(M,Z,\Pi)$ with equivariant  moment map $J\colon M\to \g^*$. Assume the action is free and proper on $J^{-1}(0)$. Then $J^{-1}(0)$ is a strong $b$-coisotropic submanifold, and its reduction $J^{-1}(0)/G$ is $b$-symplectic.
\end{cor}
 
 \begin{remark}
The fact that $J^{-1}(0)/G$ is $b$-symplectic follows already  from \cite[Theorem 3.11]{MarreroPadronOlmos}, taking $A={}^{b}TM$ there. (The hypothesis made there, that $(J_*)_x\circ \rho_x\colon {}^{b}T_xM\to \g^*$ has contant rank for all $x\in J^{-1}(0)$, is satisfied since $J^{-1}(0)$ is transverse to $Z$). In that reference the authors develop a reduction theory for level sets of arbitrary regular values $\mu\in \g^*$ satisfying the contant rank hypothesis, their statement is thus more general than the reduction statement in our Corollary \ref{reduction}.
\end{remark}

\subsubsection{\underline{Exact $b$-symplectic forms}}
\leavevmode
\vspace{0.15cm}

{As a particular case of the previous construction, we consider the $b$-symplectic analog of a well-known fact in symplectic geometry. Recall that, if a Lie group $G$ acts on an exact symplectic manifold $(M,-d\theta)$ and $\theta$ is invariant under the action, then $J:M\rightarrow\mathfrak{g}^{*}$ defined by
\begin{equation}\label{moment map}
J^{x}=-\iota_{v_{x}}\theta
\end{equation}
is an equivariant moment map for the action {(in the sense of \eqref{moment})}. For a proof, see for instance \cite[Theorem 4.2.10]{mechanics}. A similar result holds in $b$-symplectic geometry.}

\begin{lemma}[Exact $b$-symplectic forms]\label{exact}
Suppose $(M,Z)$ is a $b$-manifold with exact $b$-symplectic form $\omega=-{}^{b}d\theta$. If $\phi:G\times M\rightarrow M$ is a Lie group action preserving $Z$ and $\theta \in{}^{b}\Omega^{1}(M)$, then an equivariant moment map $J:M\rightarrow\mathfrak{g}^{*}$ is given by  $J^{x}=-\iota_{V_{x}}\theta$. {Here $V_{x}\in\Gamma\left({}^{b}TM\right)$ is the lift of the infinitesimal generator $v_{x}\in\Gamma(TM)$ under the anchor $\rho$.}
\end{lemma}
\begin{proof} Clearly  $J:M\rightarrow\mathfrak{g}^{*}$  is a smooth map.
Restricting the action to the symplectic manifold $\left(M\setminus Z,\omega|_{M\setminus Z}\right)$, we know that $G\curvearrowright \left(M\setminus Z,-d\theta|_{M\setminus Z}\right)$ admits a moment map given by $J|_{M\setminus Z}$.
Hence the equality $\Pi^{\sharp}\left(dJ^{x}\right)=v_{x}$ 
holds on the dense subset $M\setminus Z$, and as both sides are smooth on $M$, it holds on the whole of $M$.
Similarly, since $J|_{M\setminus Z}$ is equivariant, it follows that $J$ itself is equivariant.
\end{proof}

{An example of Corollary \ref{reduction} and  Lemma \ref{exact} is $b$-cotangent bundle reduction. Let us recall the picture in symplectic geometry: given an action $G\curvearrowright M$, its cotangent lift $G\curvearrowright (T^{*}M,-d\theta_{can})$ preserves the tautological one-form $\theta_{can}$ and therefore it comes with an equivariant moment map $J:T^{*}M\rightarrow\mathfrak{g}^{*}$ given by \eqref{moment map} 
\[
\langle J(\alpha_{q}),x\rangle=-\langle\alpha_{q},v_{x}(q)\rangle.
\]
Here $v_{x}$ is the infinitesimal generator of $G\curvearrowright M$ corresponding with $x\in\mathfrak{g}$.
If the action $G\curvearrowright M$ is free and proper, then symplectic reduction gives $J^{-1}(0)/G\cong T^{*}(M/G)$. Indeed, in some detail, there is a well-defined map
{
\[
\varphi:J^{-1}(0)\rightarrow T^{*}(M/G),\;
\alpha_{q}\mapsto\tilde{\alpha}_{pr(q)},
\]
where $pr:M\rightarrow M/G$ denotes the projection and
\[
\tilde{\alpha}_{pr(q)}:T_{pr(q)}(M/G)\cong\frac{T_{q}M}{T_{q}(G\cdot q)}\rightarrow\mathbb{R},\;[v]\mapsto\alpha_{q}(v).
\]
}
Since the fibers of $\varphi$ coincide with the orbits of $G\curvearrowright J^{-1}(0)$, there is an induced bijection $\overline{\varphi}:J^{-1}(0)/G\rightarrow T^{*}(M/G)$, which is in fact a symplectomorphism (see \cite[Theorem 2.2.2]{reduction}).
}

\begin{cor}[{Group actions on $b$-cotangent bundles}]\label{ex:bT*group}
Given a $b$-manifold $(M,Z)$ {and a connected Lie group $G$}, assume that $\phi: G\times M\rightarrow M$ is a free and proper action that preserves $Z$. Denote by $\Phi: G\times {}^{b}T^{*}M\rightarrow {}^{b}T^{*}M$ the $b$-cotangent lift of this action, that is
\[
\left\langle \Phi_{g}(\alpha_{q}),v \right\rangle=\left\langle \alpha_{q}, \left[{}^{b}\left(\phi_{g^{-1}}\right)_{*}\right]_{\phi_{g}(q)}v\right\rangle
\]
for $\alpha_{q}\in{}^{b}T_{q}^{*}M$ and $v\in{}^{b}T_{\phi_{g}(q)}M$. {Note that the action $\Phi$ is also free and proper, and that it preserves the hypersurface ${}^{b}T^{*}M|_{Z}$.} {The action  $\Phi$ has a canonical equivariant moment map $J$, and  
$J^{-1}(0)/G$ is canonically $b$-symplectomorphic to ${}^{b}T^{*}(M/G)$.}
\end{cor}

 \begin{proof}
Denote the infinitesimal generators of $\phi$ by $v_{x}=\rho_{M}\left(V_{x}\right)\in\mathfrak{X}(M)$  and those of $\Phi$ by $\overline{v}_{x}=\rho_{\left({}^{b}T^{*}M\right)}\left(\overline{V}_{x}\right)\in\mathfrak{X}\left({}^{b}T^{*}M\right)$, where $x\in \g$. One checks that they are related via
\begin{equation}\label{related}
\pi_{*}(\overline{v}_{x})=v_{x},
\end{equation}
where $\pi:{}^{b}T^{*}M\rightarrow M$ denotes the projection. 
Since the action $\Phi$ preserves the tautological $b$-one form $\theta\in{}^{b}\Omega^{1}\left({}^{b}T^{*}M\right)$, Lemma \ref{exact} gives an equivariant  moment map $J:{}^{b}T^{*}M\rightarrow\mathfrak{g}^{*}$ defined by $J^{x}=-\iota_{\overline{V}_{x}}\theta$. Explicitly, one has
\begin{equation}\label{momentmap}
-\left\langle J(\xi_{p}),x\right\rangle=\left(\iota_{\overline{V}_{x}}\theta\right)(\xi_{p})=\theta_{\xi_{p}}\left(\overline{V}_{x}\right)_{\xi_{p}}=\left\langle \xi_{p},\left({}^{b}\pi_{*}\right)_{\xi_{p}}\left(\overline{V}_{x}\right)_{\xi_{p}}\right\rangle=\left\langle \xi_{p},\left(V_{x}\right)_{p}\right\rangle,
\end{equation}
where the last equality uses \eqref{related} and   Lemma \ref{bder}. Denoting by $K$ the tangent distribution to the orbits of $G\curvearrowright M$ and by $\sigma:K\hookrightarrow{}^{b}TM$ the splitting of the anchor $\rho_{M}:{}^{b}TM\rightarrow TM$ obtained via Lemma \ref{lem:distrsplit} $a)$, the equality \eqref{momentmap} shows that
\begin{equation}\label{preimage}
J^{-1}(0)=\sigma(K)^{0}.
\end{equation}
We now perform reduction on $J^{-1}(0)$ as in Corollary \ref{reduction}. Because the projection map $pr:(M,Z)\rightarrow (M/G,Z/G)$ is a $b$-submersion with kernel $\text{Ker}(pr_{*})=K$, Corollary \ref{split} implies that $\text{Ker}({}^{b}pr_{*})=\sigma(K)$, and therefore
\begin{equation}\label{quotient}
{}^{b}T_{pr(q)}(M/G)\cong\frac{^{b}T_{q}M}{\sigma(K_{q})}.
\end{equation}
It is now clear from \eqref{preimage} and \eqref{quotient} that $b$-covectors in $J^{-1}(0)$ descend to $M/G$, i.e.
we get a well-defined map
{
\[
\varphi:J^{-1}(0)\rightarrow{}^{b}T^{*}(M/G),\;\alpha_{q}\mapsto\tilde{\alpha}_{pr(q)}, 
\]
where
\[
\tilde{\alpha}_{pr(q)}:{}^{b}T_{pr(q)}(M/G)\cong\frac{^{b}T_{q}M}{\sigma(K_{q})}\rightarrow\mathbb{R},\;[v]\mapsto\alpha_{q}(v).
\]
}
It is easy to check that $\varphi$ is a surjective submersion {with connected fibers}.
From symplectic geometry we know that the fibers of $\varphi$ and the orbits of the $G$-action $G\curvearrowright J^{-1}(0)$ coincide on the open dense subset $J^{-1}(0)\setminus\left(J^{-1}(0)\cap\left.{}^{b}T^{*}M\right|_{Z}\right)$ of $J^{-1}(0)$. By continuity, the corresponding tangent distributions must agree on all of $J^{-1}(0)$, and so the same holds for the foliations integrating them. Therefore, the map $\varphi$ descends to a smooth bijective $b$-map 
\[
\overline{\varphi}:J^{-1}(0)/G\rightarrow {}^{b}T^{*}(M/G).
\]
Being a bijective submersion between manifolds of the same dimension, $\overline{\varphi}$ is a diffeomorphism.
The restriction of $\overline{\varphi}$ to the complement of  $\left(J^{-1}(0)\cap\left.{}^{b}T^{*}M\right|_{Z}\right)/G$,
 endowed with the symplectic structure obtained by symplectic (i.e. coisotropic) reduction,
 is a symplectomorphism onto its image. 
 Hence, by Proposition \ref{prop:bcoisored}, $\overline{\varphi}$ is a $b$-symplectomorphism .
\end{proof}

\subsubsection{\underline{Circle bundles}}
\leavevmode
\vspace{0.15cm}

We find examples for Proposition \ref{prop:bcoisored} and Corollary \ref{reduction} by ``reverse engineering''. 

\begin{prop}\label{lem:preq}
Let $(N,\omega)$ be a $b$-symplectic manifold, which for simplicity we assume to be compact.
Let $q\colon C\to N$ be a principal $S^1$-bundle, with  connection $\theta\in \Omega^1(C)$. Denote by $\sigma\in \Omega^2(N)$  the closed 2-form satisfying $d\theta=q^*\sigma$. 
\begin{enumerate}[(i)]
\item  The following is a is $b$-symplectic manifold:
\begin{equation*}
\left(C\times I, \;\;\tilde{\omega}:=  dt\wedge p^*\theta+(t-1) p^*q^*\sigma+{}^{b}p^*\:{}^{b}q^*\omega\right).
\end{equation*}
 Here $I$ is an interval around $1$ with coordinate $t$, and $p\colon C\times I\to C$ the projection.
\item  $C\times\{1\}$ is a strong $b$-coisotropic submanifold, and the reduced $b$-symplectic manifold (as in Proposition \ref{prop:bcoisored}) is isomorphic to $(N,\omega)$.
\end{enumerate}
\end{prop}

We make a few observations about   $\tilde{\omega}$. {The summand of $\tilde{\omega}$ containing $\sigma$ is necessary to ensure that $\tilde{\omega}$ is ${}^{b}d$-closed. In the special case that $C$ is the trivial $S^1$-bundle $N\times S^1$, choosing $\theta=d\rho$ for $\rho$ the angle ``coordinate'' on $S^1$ (so $\sigma=0$), the above lemma delivers the product of the $b$-symplectic manifold  $(N,\omega)$ and of the symplectic manifold $(I\times S^1, dt\wedge \theta)$.}

In the special case that $\omega$ equals the closed 2-form $\sigma$, we have $\tilde{\omega}=d(t p^*\theta)$, which can be interpreted as the prequantization of $\sigma$ when the latter is symplectic.

\begin{remark}\label{rem:momapex}
By the above proposition, we actually recover $(N,\omega)$ by moment map reduction, as in Corollary \ref{reduction}. Indeed,  $S^1$ acts on $C\times I$ (trivially on the second factor) preserving the $b$-symplectic form $\tilde{\omega}$ (since $\theta$ is $S^1$-invariant). 
An equivariant moment map is $J(x,t)=t-1$, hence  
 $C\times \{1\} =J^{-1}(0)$.
\end{remark}

\begin{proof} (i) To check that $\tilde{\omega}$ is ${}^{b}d$-closed, notice that its first two summands can be written as $d(t p^*\theta)-p^*q^*\sigma$, which is closed since $\sigma$ is closed.

For every real number $t$ sufficiently close to $1$, $(t-1)\sigma+\omega$ is a $b$-symplectic form on $N$, so its $n$-th power (where $\dim(N)=2n$) is a nowhere-vanishing element of ${}^{b}\Omega^{2n}(N)$. This implies that $\tilde{\omega}^{n+1}$ is a 
nowhere-vanishing element of ${}^{b}\Omega^{2(n+1)}(C\times I)$, shrinking $I$ if necessary. Hence $\tilde{\omega}$ is $b$-symplectic.

(ii) Denote by  $Z\subset N$ the singular hypersurface of $\omega$. Then  the singular hypersurface of $\tilde{\omega}$ is $p^{-1}(q^{-1}(Z))\subset C\times I$, which is 
transverse to $C\times\{1\}$. Therefore the latter is a $b$-submanifold, and is coisotropic since it has codimension one. If $i\colon C\times\{1\}\to C\times I$ denotes the inclusion, then we have ${}^{b}i^*\tilde{\omega}={}^{b}q^*\omega$. One consequence is that
${}^{b}T(C\times\{1\})^{\tilde{\omega}}=\ker ({}^{b}i^*\tilde{\omega})=\ker ({{}^{b}q_*})$. 
Applying the anchor $\rho$, {we obtain that the characteristic distribution $\rho\left({}^{b}T(C\times\{1\})^{\tilde{\omega}}\right)$ of $C\times\{1\}$ is given by $\ker({q_*})$}. Since the latter has constant rank one, by Proposition \ref{prop:strongrank} we conclude that  $C\times\{1\}$ is a strong $b$-coisotropic submanifold. A second consequence is that the reduced $b$-symplectic manifold is isomorphic to $(N,\omega)$.
\end{proof}

A concrete instance of the construction of Proposition \ref{lem:preq} is the following.
\begin{cor}\label{cor:explh}
Let $h$ be any smooth function on $\CC P^1$ that vanishes transversely along a hypersurface.
On $\CC^2$ consider the differential forms 
 $\Omega:=i(d{z}_1\wedge d\conj{z}_1+d{z}_2\wedge d\conj{z}_2)$ (twice the standard symplectic form)
and $\alpha:=\conj{z}_1dz_1+\conj{z_2}dz_2$, and denote by  $r$   the radius. 

\begin{enumerate}[(i)]
 \item In a neighborhood of the unit sphere $S^3$, the following is a $b$-symplectic form:\begin{equation}\label{eq:otildeh}
  \tilde{\omega}=
 \frac{1}{r^2}\left(-1+\frac{1}{P^*h}\right)\left(-\frac{i}{r^2}(\alpha\wedge \conj{\alpha})+\Omega \right)+\Omega,
\end{equation}
where $P\colon \CC^2\setminus\{0\}\to \CC P^1$ is the projection.
\item The unit sphere $S^3$ is a strong $b$-coisotropic submanifold, and the reduced $b$-symplectic manifold is $(\CC P^1,\frac{1}{h}\sigma)$ where $\sigma$ is twice the Fubini-Study symplectic form.
\end{enumerate}
\end{cor}

\begin{remark}
The diagonal action of $S^1$ on the above neighborhood of the unit sphere $S^3$ in $\CC^2$ preserves $\tilde{\omega}$ and has moment map given by   $v\mapsto ||v||^2-1$. This follows from  Remark \ref{rem:momapex} and the proof below.
\end{remark}

\begin{proof}
On $\RR^4=\CC^2$ we consider the 1-form $\tilde{\theta}=\sum_{j=1}^2x_jdy_j-y_jdx_j$. Notice that we have $d\tilde{\theta}=2\sum_{j=1}^2dx_j\wedge dy_j=\Omega$. 
Consider the unit sphere $S^3$. 
Let ${q}\colon S^3\to \CC P^1$ be the principal bundle given by the diagonal action of $U(1)$ (the Hopf fibration). Then $\theta:=i^*\tilde{\theta}$ is
a connection 1-form on $S^3$, where $i$ is the inclusion. Then $d\theta=q^*\sigma$, where $\sigma$ is the symplectic form on $\mathbb{C}P^{1}$ obtained from $\Omega$ by coisotropic reduction. 
Consider the $b$-symplectic form $\omega:=\frac{1}{h}\sigma$ on $\CC P^1$. Applying Proposition \ref{lem:preq} to $S^3\times I \overset{p}{\to} S^3 \overset{q}{\to} \CC P^1$ yields a $b$-symplectic form $\tilde{\omega}$ on $S^3\times I$, defined by
 \begin{equation}\label{eq:tildeomabs}
 \tilde{\omega}=  dt\wedge p^*\theta+
\left(t-1+\frac{1}{p^*q^*h}\right) p^*q^*\sigma.
\end{equation}

We now make $ \tilde{\omega}$ more explicit. Denote by $p'\colon \CC^2\setminus\{0\}\to S^3$ the projection $v\mapsto v/||v||$, let $r$ denote the radius function $v\mapsto ||v||$.
Then $p'^*(\theta)=\tilde{\theta}/r^2$, since  the Euler vector field $E$ satisfies $\iota_E \tilde{\theta}=0$ and $\mathcal{L}_E (\tilde{\theta}/r^2)=0$. Hence, using $q^*\sigma=d\theta$ and $d \tilde{\theta}=\Omega$
 we obtain
$$p'^*q^*\sigma=d( \tilde{\theta}/r^2)=\frac{1}{r^2}\left(-2\frac{dr}{r}\wedge \tilde{\theta}+\Omega \right).$$
Using $\tilde{\theta}=Im(\alpha)$ and $r^2=z_1\conj{z}_1+z_2\conj{z_2}$ we get $-2\frac{dr}{r}\wedge \tilde{\theta}=-\frac{i}{r^2}(\alpha\wedge \conj{\alpha})$.
If we now use the identification $(a,t)\mapsto \sqrt{t}a$ between $S^3\times I$ and a neighborhood of $S^3$ in $\CC^2$ (so $t=r^2$), then the expression \eqref{eq:tildeomabs} becomes \eqref{eq:otildeh}.
\end{proof}

\begin{remark}
We show directly from its definition \eqref{eq:otildeh} that $\tilde{\omega}$ satisfies the transversality condition required for $b$-symplectic forms. 
As $\left(-\frac{i}{r^2}(\alpha\wedge \conj{\alpha})+\Omega \right)^{\wedge 2}$ vanishes, 
one obtains $\tilde{\omega}^{\wedge 2}=-2(1-\frac{1}{r^2}+\frac{1}{r^2}\frac{1}{P^*h})
d{z}_1\wedge d\conj{z}_1\wedge d{z}_2\wedge d\conj{z}_2$. The dual 4-vector field is thus transverse to the zero section, in a neighborhood of the unit sphere $S^3$.\end{remark}

\begin{ex}
We display an example of a function $h$ on $\CC P^1$ which vanishes on the circle $\RR P^1 \subset \CC P^1$.
The function $g:=Im(\conj{z_1}z_2)=x_1y_2-y_1x_2$ on $S^3$ is $U(1)$-invariant, hence descends to a function $h$ on $\CC P^1$, which is readily seen to vanish exactly on $\RR P^1$. It vanishes linearly there: using homogeneous the coordinate
$w:=z_2/z_1$ on the open subset $\{[z_1:z_2]:z_1\neq 0\}$ of $\CC P^1$, we have\footnote{To see this, first notice that on $S^3$ we have $\conj{z_1}z_2=(\conj{z_1}z_2)/(\conj{z_1}z_1+\conj{z_2}z_2)$,
and then divide numerator and denominator by $\conj{z_1}z_1$.} $h=\frac{Im(w)}{1+|w|^2}$, which vanishes with non-zero derivative on $\{Im(w)=0\}$.
 Since $g$ is quadratic, we have $p'^*g=g/r^2$, hence the coefficient
 $\frac{1}{r^2}\left(-1+\frac{1}{P^*h}\right)$ in equation \eqref{eq:otildeh} reads
$$ \left(-\frac{1}{r^2}+\frac{1}{Im(\conj{z_1}z_2)}\right).$$
 \end{ex}
 
\bibliographystyle{abbrv}

\begin{thebibliography}{10}

\bibitem{mechanics}
R.~Abraham and J.~Marsden.
\newblock {\em Foundations of Mechanics}.
\newblock Addison-Wesley, 2nd edition, 1987.

\bibitem{CandelConlonI}
A.~Candel and L.~Conlon.
\newblock {\em Foliations. {I}}, volume~23 of {\em Graduate Studies in
  Mathematics}.
\newblock American Mathematical Society, Providence, RI, 2000.

\bibitem{ana}
A.~Cannas~da Silva.
\newblock {\em Lectures on symplectic geometry}, volume 1764 of {\em Lecture
  Notes in Mathematics}.
\newblock Springer-Verlag, Berlin, 2001.

\bibitem{models}
A.~Cannas~da Silva and A.~Weinstein.
\newblock {\em Geometric Models for Noncommutative Algebras}, volume~10 of {\em
  Berkeley Mathematics Lecture Notes series}.
\newblock American Mathematical Society, 1999.

\bibitem{CK}
G.~Cavalcanti and R.~Klaasse.
\newblock Fibrations and log-symplectic structures.
\newblock {\em Journal of Symplectic Geometry}, 17:603--638, 2019.

\bibitem{ExamGil}
G.~R. Cavalcanti.
\newblock Examples and counter-examples of log-symplectic manifolds.
\newblock {\em J. Topol.}, 10(1):1--21, 2017.

\bibitem{GZbdefs}
S.~Geudens and M.~Zambon.
\newblock {Deformations of Lagrangian submanifolds in $b$-symplectic geometry.}.
\newblock {In preparation}.

\bibitem{Gotay}
M.~J. Gotay.
\newblock On coisotropic imbeddings of presymplectic manifolds.
\newblock {\em Proc. Amer. Math. Soc.}, 84(1):111--114, 1982.

\bibitem{GuLi}
M.~Gualtieri and S.~Li.
\newblock Symplectic groupoids of log symplectic manifolds.
\newblock {\em Int. Math. Res. Not. IMRN}, (11):3022--3074, 2014.

\bibitem{gualtieri}
M.~Gualtieri, S.~Li, A.~Pelayo, and T.~Ratiu.
\newblock The tropical momentum map: a classification of toric log symplectic
  manifolds.
\newblock {\em Math. Ann.}, 367(3-4):1217--1258, 2016.

\bibitem{miranda}
V.~Guillemin, E.~Miranda, and A.~R. Pires.
\newblock {Symplectic and Poisson geometry on b-manifolds}.
\newblock {\em Adv. Math.}, 264:864--896, 2014.

\bibitem{ToricB}
V.~Guillemin, E.~Miranda, A.~R. Pires, and G.~Scott.
\newblock Toric {A}ctions on {$b$}-{S}ymplectic {M}anifolds.
\newblock {\em Int. Math. Res. Not. IMRN}, (14):5818--5848, 2015.

\bibitem{CharlotteThesis}
C.~Kirchhoff-Lukat.
\newblock {\em {Aspects of Generalized Geometry: Branes with Boundary,
  Blow-ups, Brackets and Bundles}}.
\newblock PhD thesis, University of Cambridge, {available at
  \texttt{https://www.repository.cam.ac.uk/handle/1810/283007}}, 2018.

\bibitem{RalphThesis}
R.~Klaasse.
\newblock {\em Geometric structures and Lie algebroids}.
\newblock PhD thesis, Utrecht University, ArXiv: 1712.09560, 2017.

\bibitem{KLAlmRegPois}
R.~Klaasse and M.~Lanius.
\newblock {Poisson cohomology of almost-regular Poisson structures}.
\newblock {In preparation}.

\bibitem{MarreroPadronOlmos}
J.~C. Marrero, E.~Padr\'{o}n, and M.~Rodr\'{\i}guez-Olmos.
\newblock Reduction of a symplectic-like {L}ie algebroid with momentum map and
  its application to fiberwise linear {P}oisson structures.
\newblock {\em J. Phys. A}, 45(16):165201, 2012.

\bibitem{reduction}
J.~Marsden, G.~Misiolek, J.-P. Ortega, M.~Perlmutter, and T.~Ratiu.
\newblock {\em Hamiltonian Reduction by Stages}.
\newblock Springer, 2007.

\bibitem{Melrose}
R.~B. Melrose.
\newblock {\em The {A}tiyah-{P}atodi-{S}inger index theorem}, volume~4 of {\em
  Research Notes in Mathematics}.
\newblock A K Peters, Ltd., Wellesley, MA, 1993.

\bibitem{Milnor}
J.~Milnor.
\newblock {\em Lectures on the h-cobordism theorem}.
\newblock Princeton University Press, 1965.

\bibitem{Polishchuk}
A.~Polishchuk.
\newblock Algebraic geometry of poisson brackets.
\newblock {\em J. Math. Sci.}, 84:1413--1444, 1997.

\bibitem{SevestreWurzbacher}
G.~Sevestre and T.~Wurzbacher.
\newblock Lagrangian submanifolds of standard multisymplectic manifolds.
\newblock {\em {ArXiv:1809.11005. To appear in: Geometric and Harmonic Analysis
  on Homogeneous spaces and Applications: TJC2017; Springer Proceedings in
  Mathematics and Statistics}}, 09 2018.

\bibitem{Redsingmomap}
J.~{{\'S}niatycki} and A.~{Weinstein}.
\newblock {Reduction and quantization for singular momentum mappings}.
\newblock {\em Letters in Mathematical Physics}, 7:155--161, Mar. 1983.

\bibitem{vaisman}
I.~Vaisman.
\newblock {\em Lectures on the Geometry of Poisson Manifolds}.
\newblock Birkhauser, 1994.

\end{thebibliography}

\end{document}